\documentclass[12pt]{article}
\usepackage{amsfonts}
\usepackage{amsmath}
\usepackage{amssymb}
\usepackage{amsthm}
\usepackage[english]{babel}
\usepackage[margin=1in]{geometry}
\usepackage{graphicx}
\usepackage[colorlinks]{hyperref}
\usepackage[utf8]{inputenc}
\usepackage{mathabx}
\usepackage{mathrsfs}
\usepackage{titlesec}
\usepackage{verbatim}   
\usepackage{wasysym}
\usepackage{xcolor}
\numberwithin{equation}{section}

\colorlet{linkequation}{red}

\newcommand*{\SavedEqref}{}
\let\SavedEqref\eqref
\renewcommand*{\eqref}[1]{%
  \begingroup
    \hypersetup{
      linkcolor=linkequation,
      linkbordercolor=linkequation,
    }%
    \SavedEqref{#1}%
  \endgroup
}

\titleformat{\paragraph}
{\normalfont\normalsize\bfseries}{\theparagraph}{1em}{}
\titlespacing*{\paragraph}
{0pt}{3.25ex plus 1ex minus .2ex}{1.5ex plus .2ex}

\providecommand{\abs}[1]{\left| #1\right|}
\providecommand{\norm}[1]{\left|\left|#1\right|\right|}
\providecommand{\wt}[1]{\widetilde{#1}}
\providecommand{\wh}[1]{\hat{#1}}

\let \div \relax

\DeclareMathOperator{\tr}{tr}
\DeclareMathOperator{\div}{div}

\newcommand{\N}{\mathbb{N}}
\newcommand{\R}{\mathbb{R}}
\newcommand{\C}{\mathbb{C}}
\newcommand{\Rm}{\mathrm{Rm}}
\newcommand{\Ric}{\mathrm{Ric}}
\newcommand{\Sc}{\mathrm{scal}}
\newcommand{\dv}{\, dV}

\newcommand{\eps}{\varepsilon}
\newcommand{\Lap}{\Delta}

\newtheorem{definition}[equation]{Definition}
\newtheorem{proposition}[equation]{Proposition}
\newtheorem{theorem}[equation]{Theorem}
\newtheorem{corollary}[equation]{Corollary}
\newtheorem{lemma}[equation]{Lemma}
\newtheorem{remark}[equation]{Remark}

\title{Dynamical Stability and Instability of Poincar\'e--Einstein Manifolds}

\author{Klaus Kr\"oncke and Louis Yudowitz}
\date{\today}

\begin{document}

\maketitle
\begin{abstract}
We prove dynamical stability and instability theorems for Poincar\'{e}-Einstein metrics under the Ricci flow. Our key tool is a variant of the expander entropy for asymptotically hyperbolic manifolds, which Dahl, McCormick and the first author established in a recent article. It allows us to characterize stability and instability in terms of a local positive mass theorem and in terms of volume comparison for nearby metrics.
\end{abstract}
\tableofcontents

\section{Introduction}
    Let $N^n$ be a compact manifold with boundary $\partial N$ and  $M = N \setminus \partial N$ be its interior. Recall that a Riemannian metric $g$ on $M$ is called conformally compact of class $C^{k,\alpha}$, if there is a boundary defining function $\rho$ such that $\rho^2g$ extends to a $C^{k,\alpha}$-Riemannian metric $h$ on $N$. Here, a function $\rho\in C^{\infty}\left(N,\left[0,\infty\right)\right)$ is called boundary defining if $\rho^{-1}\left(0\right) = \partial N$ and $\left.d\rho\right|_p\neq 0$ for all $p\in \partial N$. If $\abs{d\rho}^2_{h} \equiv 1$, we say that $\left(M,g\right)$ is {\em asymptotically hyperbolic}, or simply AH.
\pagebreak

    An Einstein manifold $\left(M^n,g\right)$ which is asymptotically hyperbolic is called Poincar\'{e}-Einstein (PE). Note that in this case, the Ricci curvature necessarily satisfies $\Ric=-\left(n-1\right)g$ because the sectional curvatures converge to $-1$ at infinity. PE metrics form perhaps the most interesting class of AH metrics and play an important role in the AdS-CFT correspondence (see e.g.\ \cite{biquard_AdS}).

    In this paper, we consider the dynamical stability of PE metrics as stationary points of the Ricci flow
    \begin{align}\label{AH_Ricci_flow}
        \partial_tg=-2\Ric_{g}-2\left(n-1\right)g,
    \end{align}
    where the additional constant in the flow equation is the proper normalization in this setting. As a first example, the dynamical stability of hyperbolic space has been established by Schulze--Schn\"{u}rer--Simon \cite{SSS11}, a result which has been refined and generalized to symmetric spaces of noncompact type by Bamler \cite{bamler_symm_spaces}. In \cite{volume_comparison}, Hu--Ji--Shi prove the dynamical stability of a strictly linearly stable PE manifold. Additionally, they conclude a volume comparison result for nearby metrics, for which we use the following definition:
    \begin{definition}
        Let $\left(M,\hat{g}\right)$ be a Poincar\'{e}-Einstein manifold of class $C^{2,\alpha}$. We say that $\hat{g}$ admits the {local volume comparison property} if for all nearby metrics $g$ with $\Sc_g\geq -n\left(n-1\right)$ and $g-\wh{g}$ decaying sufficiently fast, the renormalized volume $RV_{\hat{g}}\left(g\right)"="\mathrm{vol}\left(g\right)-\mathrm{vol}\left(\wh{g}\right)$ is nonnegative.
    \end{definition}

    Here, $\Sc_g$ is the scalar curvature of $g$. Now, Hu--Ji--Shi use their dynamical stability result to show that strictly linearly stable PE manifolds satisfy the local volume comparison property.
    In this paper, we are going to relax the assumptions for dynamical stability.
    Furthermore, we prove a converse of the above implication, i.e.\ we will conclude dynamical stability from the local volume comparison property.

    Our key quantity is the expander entropy $\mu_{AH,\wh{g}}$ for the Ricci flow in the asymptotically hyperbolic setting, which was introduced by Dahl, McCormick and the first author in \cite{poincare_einstein_entropy}. The functional $\mu_{AH,\wh{g}}$ is nondecreasing along the Ricci flow \eqref{AH_Ricci_flow} and constant only along PE metrics.

    The expander entropy is closely related to the volume-renormalized mass $m_{VR,\hat{g}}\left(g\right)$ of a metric $g$ relative to $\hat{g}$,  a new mass-like invariant which Dahl, McCormick and the first author also introduced in \cite{poincare_einstein_entropy}. In general, one can not expect a positive mass theorem for any PE metric as the reference metric $\hat{g}$. But one may have a local positive mass theorem in the following sense:
    \begin{definition}
        Let $\left(M,\wh{g}\right)$ be a Poincar\'{e}-Einstein manifold of class $C^{2,\alpha}$. We say that $\wh{g}$ admits the {local positive mass property} if for all nearby metrics $g$ with $\Sc_g \geq -n\left(n-1\right)$, the volume-renormalized mass $m_{VR,\wh{g}}\left(g\right)$ is nonnegative.
    \end{definition}
     The following theorem, partly proven in \cite[Theorem D]{poincare_einstein_entropy}, relates the local behavior of the entropy and the renormalized volume to a local positive mass theorem for the volume-renormalized mass. The topology on the space of metrics which we use for the theorem below as well as for the two definitions above is the $H^k$-topology for $k > \frac{n}{2} + 2$, where $n$ is the dimension of the manifold.

    \begin{theorem}\label{mainthm_local_positive_mass}
        Let $\left(M,\hat{g}\right)$ be a Poincar\'{e}-Einstein manifold of class $C^{2,\alpha}$. The following are equivalent:
        \begin{itemize}
            \item[(i)] $\wh{g}$ is a local maximizer of $\mu_{AH,\wh{g}}$
            \item[(ii)] $\wh{g}$ satisfies the local positive mass property
            \item[(iii)] $\wh{g}$ satisfies the local volume comparison property
        \end{itemize}
    \end{theorem}
    A more detailed statement is given in Theorem \ref{thm_local_positive_mass} below. As discussed in \cite[p.\ 5-6]{poincare_einstein_entropy}, this solves an analogue of Ilmanen's conjecture (which Ilmanen formulated for Ricci-flat manifolds which are asymptotically locally Euclidean, or ALE for short) in the asymptotically hyperbolic setting. Now we are ready to state the two main results of this paper.
    \begin{theorem}\label{mainthm_stability} (Dynamical stability)
        Let $\left(M,\hat{g}\right)$ be a Poincar\'{e}-Einstein manifold of class $C^{2,\alpha}$. If $\hat{g}$ is a local maximizer of $\mu_{AH,\hat{g}}$, then
        for every $L^2 \cap L^{\infty}$-neighborhood $\mathcal{U}$ of $\hat{g}$, there exists a $L^2\cap L^{\infty}$-neighborhood $\mathcal{V} \subset \mathcal{U}$
        such that the Ricci flow starting at any metric in  $\mathcal{V}$ exists for all times and converges polynomially
        (modulo diffeomorphisms)  to a Poincar\'{e}-Einstein  metric in $\mathcal{U}$.
    \end{theorem}
    \begin{theorem}\label{mainthm_instability} (Dynamical instability)
        Let $\left(M,\hat{g}\right)$ be a Poincar\'{e}-Einstein manifold of class $C^{2,\alpha}$. If $\hat{g}$ is not a local maximizer of $\mu_{AH,\hat{g}}$, then there exists a nontrivial ancient Ricci flow $\left\{g\left(t\right)\right\}_{t\in\left(-\infty,0\right]}$ that converges (modulo diffeomorphisms) to $\hat{g}$ as $t \to -\infty$.
    \end{theorem}
    Observe also that the converse implications also hold due to the monotonicity of $\mu_{AH,\hat{g}}$ along the Ricci flow. In particular, each PE manifold is either dynamically stable or dynamically unstable and this entirely depends on the local behavior of $\mu_{AH,\hat{g}}$.

    Summarizing and combining these results and the discussion, we establish the following equivalences for Poincar\'{e}-Einstein manifolds:
    \begin{align*}
        \text{dynamical stability}\Leftrightarrow \text{positive mass theorem for nearby metrics}\\
        \text{dynamical instability}\Leftrightarrow \text{failure of positive mass  for nearby metrics}
    \end{align*}
    The statements and proofs of Theorem \ref{mainthm_stability} and \ref{mainthm_instability} are similar to corresponding results in the compact case, see \cite{HM14} by Haslhofer-M\"uller for Ricci-flat manifolds and \cite{kroncke_Ricci_solitons,kroncke_Einstein} by the first author for Einstein metrics and Ricci solitons. An essential ingredient is a \L ojasziewicz-Simon inequality for  $\mu_{AH,\hat{g}}$ which is as follows:

    \begin{theorem}\label{mainthm_LS_ineq}  (\L ojasiewicz-Simon inequality)
        Let $\left(M,\wh{g}\right)$ be a Poincar\'e--Einstein manifold  of class $C^{2,\alpha}$. Then there exists a positive constant $C > 0$ and an exponent $\theta \in \left(0,1\right]$ such that we have
        \begin{equation}
            \abs{\mu_{\mathrm{AH},\wh{g}}\left(g\right) - \mu_{\mathrm{AH},\wh{g}}\left(\wh{g}\right)}^{2-\theta} \leq C\norm{\nabla\mu_{\mathrm{AH},\wh{g}}\left(g\right)}^2_{L^2\left(M\right)}
        \end{equation}
        for all metrics $g$ sufficiently close to $\hat{g}$.
    \end{theorem}

    The reason that these results and proofs are similar to the compact case lies in the behavior of the linearization of $\nabla\mu_{\mathrm{AH},\wh{g}}$. Both in the compact case and in the asymptotically hyperbolic case, it is (up to gauge) Fredholm as a map between (unweighted) Sobolev spaces and positive on the orthogonal complement of the kernel if it is nonnegative.

    In the case of ALE manifolds, the Fredholm property only holds between spaces of specific weights. Using this, Deruelle-Ozuch \cite[Theorem 0.3]{ale_loj} establish a \L ojasziewicz-Simon inequality for a suitably weighted $L^2$-norm which only under additional assumptions extends to an inequality of the standard $L^2$-norm.

    The present paper is structured as follows. In Section \ref{sec_Fredholm}, we establish some Fredholm and isomorphism properties for Laplace-type operators on PE manifolds. In Section \ref{sec_entropy}, we briefly review the expander entropy and some of its properties. In the Sections \ref{sec_pmt} and \ref{loj_section}, we prove Theorem \ref{mainthm_local_positive_mass} and Theorem \ref{mainthm_LS_ineq}, respectively. Finally, we prove Theorems \ref{mainthm_stability} and \ref{mainthm_instability} in Section \ref{stability_thms}.
\section{Fredholm Operators on Poincar\'e--Einstein Manifolds}\label{sec_Fredholm}
    In this section, we are going to show that for PE manifolds, some Fredholm and isomorphism properties for elliptic operators are independent from the boundary regularity. The following lemma is a consequence of results in \cite{lee_book}. Our sign convention for the rough Laplacian on tensor fields is $\Delta=\nabla^*\nabla=-\mathrm{tr}\nabla^2$.
    \begin{lemma}\label{low_reg_isomorphism_result}
        Let $\left(M,\wh{g}\right)$ be a conformally compact manifold of class $C^{2,\alpha}$ and let $E \rightarrow M$ be an $\left(r_1,r_2\right)$-tensor bundle over $M$. Then there exists a constant $\lambda_1 = \lambda_1\left(n,r_1,r_2\right)$ such that for all $c > \lambda_1$, the operator $\Lap + c: C^{2,\alpha}\left(M;E\right) \rightarrow C^{0,\alpha}\left(M;E\right)$ acting on $\left(r_1,r_2\right)$-tensors is an isomorphism. Furthermore, there exists a constant $\lambda_2 = \lambda_2\left(n,r_1,r_2\right)$ such that for all $c > \lambda_2$, the operator $\Lap + c: H^{2}\left(M;E\right) \rightarrow L^{2}\left(M;E\right)$ acting on $\left(r_1,r_2\right)$-tensors is an isomorphism.         
    \end{lemma}
    \begin{proof}
        We are first going to prove that for each $\left(r_1,r_2\right)$-tensor bundle $E$ and each $R_0>0$, there exists a $\mu=\mu\left(n,r_1,r_2,R_0\right)$ such that, for all $c>\mu$, the indicial radius $R$ of the operator $\Delta + c$ satisfies $R>R_0$.
        For convenience for the reader, we recall the definition of an indicial radius (c.f\ also \cite{lee_book}).

        For each $s \in \C$, a uniformly degenerate operator $P$ induces a bundle endomorphism $I_s\left(P\right)$ on the boundary, given by $I_s(P)(\overline{u})=\rho^{-s} P(\rho^su)|_{\partial M}$, where $u\in C^{\infty}(M,E)$ is an arbitrary extension of $\overline{u}\in C^{\infty}(\partial M,E|_{\partial M})$. Now $s$ is called a characteristic exponent of $P$ at $p\in \partial M$, if $I_s(P)|_{p}:E_p\to E_p$ is singular. By \cite[Lemma 4.3]{lee_book}, the characteristic exponents are constant on $\partial M$ as long as $\partial M$ is connected.
         
        If $E$ is an $(r_1,r_2)$ tensor bundle, we call $r=r_1-r_2$ the weight of the tensor bundle. Now, the indicial radius $R$ is the smallest nonnegative number such that $P$ has a characteristic exponent $s$ with $\mathrm{Re}(s)=\frac{n-1}{2}+R-r$. By \cite[Proposition 7.3]{lee_book}, we get 
        \begin{equation*}
            I_s(\Delta+c)=I_0(\Delta+c)+s(n-1-s-2r)=I_0(\Delta)+c+s(n-1-s-2r).
       \end{equation*} 
      Thus a complex number $s$ of the form $s=\frac{n-1}{2}+R-r$ is a characteristic exponent of $\Delta+c$ if and only if the operator
        \begin{equation*}
            A:=I_0(\Delta)+c+\left(\frac{n-1}{2}-r\right)^2-R^2
        \end{equation*}
        is singular. Obviously, this can only be the case if $R^2$ is a real number. Let $i_0$ be the lower bound for $I_0(\Delta)$ and $\overline{u}\in \mathrm{ker}(A)$. Then, 
        \begin{equation}\label{singular_operator}
            0=\langle A\overline{u},\overline{u}\rangle \geq i_0|\overline{u}|^2+c|\overline{u}|^2+\left(\frac{n-1}{2}-r\right)^2|\overline{u}|^2 - R^2|\overline{u}|^2.
        \end{equation}
        Now for given $R_0>0$, choose $\mu\in\R$ such that
        \begin{equation*}
            i_0+\mu +\left(\frac{n-1}{2}-r\right)^2=(R_0)^2.
        \end{equation*}
        Then for all $c>\mu$ and $R\in \C$ with $R^2\in\R$ and $R^2\leq (R_0)^2$ (hence all $R$ some are either purely imaginary or real with $|R|\leq R_0$), the right hand side of \eqref{singular_operator} is positive as long as $\overline{u}$ is nontrivial. This implies that the operator $A$ is nonsingular for all such values of $c$ and $R$. Thus, the indicial radius of $\Delta+c$ is larger than $R_0$, proving the claim.

        Choose $\lambda_1>0$ such that $\Delta+c$ has positive indicial radius  larger than $\frac{n-1}{2}$  for $c>\lambda_1$ and $\lambda_2>0$ such that $\Delta+c$ has positive indicial radius for $c>\lambda_2$.
        Then by \cite[Theorem C]{lee_book}, the operator $\Delta+c:C^{2,\alpha}\to C^{0,\alpha}$ is Fredholm of index zero if $c>\lambda_1$ and $\Delta+c:H^2\to L^2$ is Fredholm of index zero if $c>\lambda_2$. Furthermore, in both cases, the kernel is equal to the $L^2$ kernel, which is trivial because $\lambda_i>0$ for $i=1,2$. Thus the operators are isomorphisms, as desired.
    \end{proof}
             The next lemma shows that in the case of PE manifolds, the isomorphism property extends to function spaces of higher regularity, independently of the regularity of the conformal compactification.
    \begin{lemma}\label{isomorphism_result}
        Let $\left(M,\wh{g}\right)$ be a PE manifold of class $C^{2,\alpha}$. Let $E \rightarrow M$ be an $\left(r_1,r_2\right)$-tensor bundle over $M$ and $\lambda_i$ be as in Lemma \ref{low_reg_isomorphism_result}.
        
        Then for all $c > \lambda_1$ and all $k\in\N_0$, the operator $\Lap + c: C^{k+2,\alpha}\left(M;E\right) \rightarrow C^{k,\alpha}\left(M;E\right)$ acting on $\left(r_1,r_2\right)$-tensors is an isomorphism. Similarly, for all $c>\lambda_2$ and all $k \in \N_0$, the operator $\Lap + c : H^{k+2}\left(M;E\right) \rightarrow H^k\left(M;E\right)$ is an isomorphism.
    \end{lemma}
    %
    \begin{proof}    To ease notation slightly, We will omit in this lemma the dependence on the tensor bundle $E$. 
        By Lemma \ref{low_reg_isomorphism_result},
        \begin{align*}
            \Lap + c : C^{2,\alpha} &\rightarrow C^{0,\alpha}
        \end{align*}
        is an isomorphism for $c>\lambda_1$. We fix the constant $c$ throughout the proof.
        We want to show that the restriction
        \begin{equation*}
            \Lap + c : C^{k+2,\alpha}\rightarrow C^{k,\alpha}
        \end{equation*}
        is an isomorphism for all $k\in\N$. This map is obviously injective as the restriction of an injective map. Thus it remains to show surjectivity. Consider $\varphi\in C^{0,\alpha}$. We know already that there exists $u\in C^{2,\alpha}$ such that
        \begin{equation}\label{surj_pde}
            \Lap u + cu = \varphi.
        \end{equation}
        We are going to show by induction on $k$ that if $\varphi\in C^{k,\alpha}$, then $u\in C^{k+2,\alpha}$ and 
        \begin{equation}\label{ind_estimate}
            \norm{u}_{C^{k+2,\alpha}}\leq  C\norm{\varphi}_{C^{k,\alpha}},
        \end{equation} 
        thereby establishing the desired surjectivity. We know that this inequality holds for $k=0$. Note that, because the manifold is $C^{2,\alpha}$-conformally compact, $\norm{\Rm}_{C^{0,\alpha}}\leq C<\infty$. In order to improve the regularity of $u$, we will also have to improve the regularity of $\Rm$. For this, note that the Riemann curvature tensor $\Rm$ of $\wh{g}$ satisfies the following elliptic PDE when $\left(M,\wh{g}\right)$ is an Einstein manifold:
        \begin{equation}\label{rm_evo}
            \Lap \Rm = \Rm + \Rm \ast \Rm.
        \end{equation}
                By Lemma \ref{low_reg_isomorphism_result}, this immediately gives us $\norm{\Rm}_{C^{2,\alpha}}\leq C$.
                
  Now we are ready to do the induction step.
        Assume that
        \begin{equation}
            \norm{u}_{C^{k+1,\alpha}}\leq  C\norm{\varphi}_{C^{k-1,\alpha}}, \qquad \norm{\Rm}_{C^{k+1,\alpha}}\leq C.
        \end{equation}
        We differentiate \eqref{surj_pde} $k$ - times covariantly and commute derivatives to get
        \begin{equation}\label{differentiated_identity}
            \nabla^k \varphi = \nabla^k\left(\Lap + c\right)u = \left(\Lap + c\right)\nabla^k u + \sum^k_{i=0}\nabla^{k-i} \Rm \ast \nabla^i u.
        \end{equation}
        Since we know $\Lap + c_k : C^{2,\alpha} \rightarrow C^{0,\alpha}$ is an isomorphism (for a $c_k$ potentially different from $c$, because the operator is acting on a different tensor bundle), we have the following:
        \begin{align}
            \norm{\nabla^k u}_{C^{2,\alpha}} &\leq C\norm{\left(\Lap + c_k\right)\nabla^k u}_{C^{0,\alpha}} \nonumber \\
            &\leq C\norm{\left(\Lap + c\right)\nabla^k u}_{C^{0,\alpha}} C|c-c_k|\norm{\nabla^k u}_{C^{0,\alpha}} \nonumber \\
            &\leq C'' \left(\norm{\nabla^k\varphi}_{C^{0,\alpha}} + (1+\norm{\Rm}_{C^{k,\alpha}\left(M\right)})\norm{u}_{C^{k,\alpha}}\right) \label{isomorphism_induction}\\
            &\leq \left(C''' + \norm{\Rm}_{C^{k,\alpha}}\right)\norm{\varphi}_{C^{k,\alpha}}, \nonumber
        \end{align}
        where the final inequality follows from the induction hypothesis. This yields the desired bound \eqref{ind_estimate} for $u$. Applying the above procedure furthermore to the equation \eqref{rm_evo} and applying the induction assumptions on the curvature yields the bound $\norm{\Rm}_{C^{k+2,\alpha}}\leq C$ and thereby finishes the induction step. The argument for the $H^k$ spaces is completely analogous, with the added benefit that we do not need to improve the regularity of the curvature tensor again.
    \end{proof}

    Combining with \cite[Proposition E]{lee_book}, this yields

    \begin{corollary}\label{cor_Lap_iso}
        Let $\left(M,\wh{g}\right)$ be a PE manifold of class $C^{2,\alpha}$. Then for all $c > 0$ and all $k \in \N_0$, the operator $\Lap + c$ acting on functions is an isomorphism both as a map  $\Lap + c : H^{k+2}\left(M\right) \rightarrow H^k\left(M\right)$ and  $\Lap + c : C^{k+2,\alpha}\left(M\right) \rightarrow  C^{k,\alpha}\left(M\right)$.
    \end{corollary}

    An operator of particular interest is the Einstein operator

    \begin{align*}
        \Lap_E:=\Lap-2\mathring{R},\qquad (\mathring{R}h)_{ij}:=R_{iklj}h^{kl},
    \end{align*}

    acting on symmetric $2$-tensors, as it appears in the second variation of the expander entropy defined below. On an Einstein manifold of constant $-n+1$, it is related to the well-known Lichnerowicz Laplacian $\Delta_L$ bt the simple identity $\Lap_E=\Delta_L+2(n-1)$.

    \begin{remark}\label{fredholm_remark}
        By \cite[Proposition D]{lee_book}, the indicial radius of the Einstein operator
        is $\frac{n-1}{2}$ and thus by \cite[Theorem C]{lee_book}, $\Lap_E:H^2\to L^2$
        is a Fredholm operator of index zero. By very similar arguments as in the proof of Lemma \ref{isomorphism_result}, one shows that it restricts to a Fredholm operator $\Lap_E:H^{k+2}\to H^k$ for any $k\in\N$, even if the regularity of the conformal compactification is just $C^{2,\alpha}$.
    \end{remark}

    \section{The Expander Entropy for Asymptotically Hyperbolic Manifolds}\label{sec_entropy}
    We will need the  expander entropy from \cite{poincare_einstein_entropy} to be analytic on an $H^k$-neighborhood of a PE metric $\mu_{\mathrm{AH},\wh{g}}$. This section is devoted to detailing the necessary modifications to the arguments in Section $5$ of \cite{poincare_einstein_entropy} to ensure this. Furthermore, we are going to establish a suitable version of the Ebin--Palais slice theorem.
  
    Throughout this section, let $(M^n,\wh{g})$ be a fixed PE manifold of class $C^{2,\alpha}$ and $k>n/2+2$. We write $\mathcal{R}^k\left(M,\wh{g}\right)$ for the space of metrics $g$ which are $H^k$-close to $\wh{g}$:

    \begin{equation*}
        \mathcal{R}^k\left(M,\wh{g}\right) := \left\{g : g > 0, g - \wh{g} \in H^k\left(M\right)\right\}.
    \end{equation*}
 For a large ball ${B}_R$ in $M$, define
    \begin{align*}
        m_{{ADM},\wh{g}}(g,R)
        &= \int_{\partial {B}_R}(\mathrm{div}_{\wh{g}}(g)-d\mathrm{tr}_{\wh{g}}(g))(\nu_{\wh{g}})\dv_{\wh{g}}
        ,\\
        RV_{\wh{g}}(g,R)
        &= \int_{B_R}\dv_g-\dv_{\wh{g}}.
    \end{align*}
    Let us recall the definition of the entropy in \cite[Section 5]{poincare_einstein_entropy}. For $f\in C^{\infty}_c(M)$, let

    \begin{align*}
        \mathcal{W}_{{AH},\wh{g}}(g,f)
        &=\lim_{R \to \infty}\Big(\int_{B_R} \Big(
        \left( |\nabla f|^2+\Sc_g+f \right)e^{-f}-2(n-1) \left( (f+1)e^{-f} - 1 \right)\Big)\dv_g\\
           &\qquad\qquad - m_{{\rm ADM},\wh{g}}(g,R) - 2(n-1)RV_{\wh{g}}(g,R)\Big).
    \end{align*} 
    \begin{definition}
        The {\em expander entropy} of $g$ (relative to  $\wh{g}$) is defined as
        \begin{align*}
            \mu_{{AH},\wh{g}}(g)
            =\inf_{ f\in C_c^{\infty}(M)}\mathcal{W}_{{AH},\wh{g}}(g,f).
        \end{align*}
    \end{definition}
    It can be shown, exactly as in \cite[Section 5]{poincare_einstein_entropy}, that for each $g\in   \mathcal{R}^k\left(M,\wh{g}\right)$, there is a unique minimizer in the definition of the entropy which satisfies the Euler-Lagrange equation

    \begin{align}\label{el_equation_entropy}
        2\Delta f_g +|\nabla f_g|^2-\Sc_g-n(n-1)+2(n-1)f_g=0.
    \end{align}

    In \cite{poincare_einstein_entropy}, the details are carried out for the case where $g-\wh{g}\in C^{k,\alpha}_{\delta}$ for a weight $\delta>\frac{n-1}{2}$. The arguments for Sobolev regularity are exactly the same, up to replacing H\"{o}lder spaces by Sobolev spaces at the right places. Applying the implicit function theorem to \eqref{el_equation_entropy} and using Corollary \ref{cor_Lap_iso} shows that the minimizer depends analytically on the metric, and hence $\mathcal{R}^k\left(M,\wh{g}\right)\ni g\mapsto \mu_{AH,\wh{g}}(g)$ is analytic as well (c.f.\ \cite[Proposition $5.16$]{poincare_einstein_entropy} for the H\"{o}lder regularity case). The first variation, of the entropy is given by
    \begin{align*}
        \left.\frac{d}{dt}\right|_{t=0}\mu_{{AH},\wh{g}}(g+th)=-\int_M \langle \Ric+\nabla^2f_g+(n-1)g,h\rangle e^{-f_g}\dv.
    \end{align*}
    See \cite[Proposition 5.17]{poincare_einstein_entropy} for the full derivation. Consequently, $\mu_{{AH},\wh{g}}$ is nondecreasing along the Ricci flow
    \begin{align*}
     \partial_t g_t=-2\Ric_{g_t}-(n-1)g_t
    \end{align*}
    and constant if and only if $g_t$ is a constant curve of PE metrics. Expanding Ricci solitons are excluded due to an integration by parts argument, see \cite[Corollary 5.18]{poincare_einstein_entropy}. For further purposes, we need the following slice theorem:

    \begin{proposition}\label{entropy_analytic_extension}
        Let
           \begin{equation*}
            \mathcal{S}_{\wh{g}} = \left\{g \in \mathcal{R}^k\left(M, \wh{g}\right) : \div_{\wh{g}} g = 0\right\}.
        \end{equation*}
        Then there is a neighborhood $\mathcal{U} \subset \mathcal{R}^k\left(M,\wh{g}\right)$ of $\wh{g}$ such that any $g \in \mathcal{U}$ can uniquely be written as $g = \varphi^\ast \wt{g}$ for some $\wt{g} \in \mathcal{U} \cap \mathcal{S}_{\wh{g}}$ and a diffeomorphism $\varphi \in \mathrm{Diff}^{k+1}\left(M\right)$ that is $H^{k+1}$-close to the identity. 
    \end{proposition}
    \begin{proof}
        This is a rephrasing of Lemma $6.1$ in \cite{poincare_einstein_entropy}, which we refer the reader to for details. The proof of the result holds after the weighted H\"older spaces in \cite{poincare_einstein_entropy} are replaced with unweighted Sobolev spaces.
    \end{proof}
\section{Positive Mass and Volume Comparison}\label{sec_pmt}
    Let $(M^n,\wh{g})$ and $k$ be as before.
    Given   $g\in \mathcal{R}^k\left(M,\wh{g}\right)$, the volume-renormalized mass of $g$ with respect to $\wh{g}$ is
    \begin{equation} \label{eq-first-m-defn}
        m_{{VR},\wh {g}}(g)= \lim_{R \to \infty}\left( m_{{ADM},\wh{g}}(g,R)+2(n-1)RV_{\wh{g}}(g,R) \right).
    \end{equation}
    We have the following theorem (c.f.\ \cite[Theorem 3.1]{poincare_einstein_entropy}):
    \begin{theorem}
        For $g\in \mathcal{R}^k\left(M,\wh{g}\right)$, the limit
        \begin{align*}
            S_{\wh{g}}(g) = \lim_{R \to \infty}
        \Big(& \int_{B_R} \left( \Sc_g + n(n-1) \right) \dv_g -m_{{ADM},\wh{g}}(g,R) 
        - 2(n-1)RV_{\wh{g}}(g,R)\Big)
        \end{align*}
        is well-defined and finite. In particular, $m_{{ VR},\wh{g}}(g)$ is well-defined and finite if $\Sc_g + n(n-1)\in L^1(M)$. 
    \end{theorem}
    \begin{proof}
        The corresponding version in \cite{poincare_einstein_entropy} is stated for weighted H\"{o}lder spaces and we briefly outline a proof adapted to the present situation. If $h=g-\hat{g}$, Then Taylor expansion yields
        \begin{align*}
            \Sc_g + n(n-1)&=D\Sc_{\hat{g}}[h]+O(h^2)=\div_{\hat{g}}(\div_{\hat{g}} h -d\mathrm{tr}_{\hat{g}}h)-\langle \Ric_{\hat{g}},h\rangle+O(h^2)\\
            \dv_g&=\dv_{\hat{g}}+D\dv_{\hat{g}}[h]+O(h^2)=\dv_{\hat{g}}+\frac{1}{2}\mathrm{tr}_{\hat{g}}h\dv_{\hat{g}}+O(h^2)
        \end{align*}
        and, because $\langle\Ric_{\hat{g}},h\rangle=-(n-1)\mathrm{tr}_{\hat{g}}h$, we can combine these expansions to
        \begin{align*}
            (\Sc_g + n(n-1))\dv_g-\div_{\hat{g}}(\div_{\hat{g}} h -d\mathrm{tr}_{\hat{g}}h)\dv_{\hat{g}} -2(n-1)(\dv_g-\dv_{\hat{g}})=O(h^2)\dv_{\hat{g}}.
        \end{align*}
        The error term is of the form
        \begin{align*}
            O(h^2)=\sum_{i+j\leq 2}\nabla^ih*\nabla^jh,
        \end{align*}
        so that the right hand side is integrable under the assumptions of the theorem.
        Thus integration over a large ball $B_R$, applying the divergence theorem, and letting $R\to\infty$ yields the result.
    \end{proof}
    The $L^2$-gradient of this functional is given by
      \begin{align*}
        \nabla S_{\wh{g}}=-\Ric_ g+\frac{1}{2}\Sc_g\cdot g+\frac{1}{2}(n-1)(n-2)g,
    \end{align*}
    which implies that $g \mapsto S_{\wh{g}}(g)$ is analytic functional on $H^k$ whose critical points are precisely the PE metrics.

    Finally, we want to introduce the renormalized volume. Let $g\in \mathcal{R}^k\left(M,\wh{g}\right)$ such that $g-\wh{g}=O_1(e^{-\delta r})$ for some $\delta>n-1$, i.e.\
    \begin{align*}
    |g-\wh{g}|_{\wh{g}}+|\nabla^{\wh{g}} (g-\wh{g})|_{\wh{g}}\leq Ce^{-\delta r}
    \end{align*}
   for some $C=0$. 
    Then, $m_{{ADM},\wh{g}}(g,R)\to 0$ at infinity and we have a well-defined limit
    \begin{align*}
        RV_{\wh{g}}(g) = \lim_{R \to \infty} RV_{\wh{g}}(g,R).
    \end{align*}
    Therefore, 
    \begin{align*}
        S_{\wh{g}}(g)=\int_{M} \left( \Sc_g + n(n-1) \right) \dv_g+2(n-1)RV_{\wh{g}}(g).
    \end{align*}
    In particular, $\Sc_g + n(n-1)\in L^1$ and the volume-renormalized mass reduces to a multiple of the renormalized volume. We now can prove the following theorem:
    \begin{theorem}\label{thm_local_positive_mass}
        Let $(M,\wh{g})$ be a PE manifold and $\delta\in (n-1,n)$. Then the following are equivalent:
        \begin{itemize}
            \item[(i)] For all metrics $g$ sufficiently close to $\wh{g}$, we have $\mu_{AH,\wh{g}}(g)\leq 0$.
            \item[(ii)] For all metrics $g$ sufficiently close to $\wh{g}$ with $\Sc_g+n(n-1)$ being nonnegative and integrable, we have $m_{{VR},\wh{g}}(g)\geq 0$. 
            \item[(iii)]  For all metrics $g$ sufficiently close to $\wh{g}$ with $g-\wh{g}=O_1(e^{-\delta r})$ and $\Sc_g+n(n-1)\geq0,$ we have $RV_{\wh{g}}(g)\geq 0$.
        \end{itemize}
    \end{theorem}
    \begin{proof}
        The equivalence $(i)\Leftrightarrow(ii)$ was essentially proven in \cite[Theorem 6.8]{poincare_einstein_entropy}, with the only difference again that  we use the $H^k$-topology with $k>\frac{n}{2}+2$ for the assertion here instead of the $C^{k,\alpha}$-topology. We therefore omit the details here.

        Observe that the implication $(ii)\Rightarrow (iii)$ is obvious, since $m_{{VR},\wh{g}}(g)=2(n-1)RV_{\wh{g}}(g)$ and $\Sc_g+n(n-1)\in L^1$ if $g-\wh{g}=O_1(e^{-\delta r})$. 

        Now we are going to prove $(iii)\Rightarrow (ii)$. Suppose that there is a metric  $g$ with $\Sc_{g}\geq-n(n-1)$ which is $H^k$-close to $\wh{g}$ with $m_{{VR},\wh{g}}(g)<0$. Since the constant scalar curvature metric in the conformal class will have smaller mass (c.f.\ \cite[Theorem 4.5]{poincare_einstein_entropy}), we may assume that $\Sc_{g}=-n(n-1)$. Let $g_i$, $i\in\N$, be a sequence of metrics converging to $g$ in $H^k$ with $g_i-\wh{g}\in C^{\infty}_{c}$. Let $\overline{g}_i=e^{2w_i}g_i$ be the conformal metric of constant scalar curvature $-n(n-1)$. Then $\overline{g}_i\to g$ in $H^k$. 
    
        A bootstrapping argument for the equation on the conformal factor (c.f.\ the proof of \cite[Theorem 4.9]{poincare_einstein_entropy}) shows that $w_i\in O_1(e^{-\delta r})$, so that $\overline{g}_i-\hat{g}\in O_1(e^{-\delta r})$ and

        \begin{align*}
            2(n-1)RV_{\wh{g}}(\overline{g}_i)=m_{{VR},\wh{g}}(\overline{g}_i)=-S_{\wh{g}}(\overline{g}_i)\to -S_{\wh{g}}(g)= m_{{VR},\wh{g}}(g)<0.
        \end{align*}

        Here, we used that $\overline{g}_i$ and $g$ have constant scalar curvature $-n(n-1)$ and that $S_{\wh{g}}$ is an analytic functional on $H^k$. Hence, $RV_{\wh{g}}(\overline{g}_i)<0$ for sufficiently large $i$.

        If now $g_j$, $j\in\N$, was a sequence of constant scalar curvature metrics converging to $\wh{g}$ with $m_{{VR},\wh{g}}(g_i)<0$, we can approximate as above each of the $g_i$ with a sequence of constant scalar curvature metrics $\overline{g}_{i,j}$ with $\overline{g}_{i,j}-\wh{g}\in O_1(e^{-\delta r})$. Because for each fixed $i$, we obtain $m_{{VR},\wh{g}}(g_{i,j})<0$ for sufficiently large $j$, we can construct a subsequence $g_{i,j(i)}$, $i\in\N$ of the double sequence with $m_{{VR},\wh{g}}(g_{i,j(i)})<0$ which converges to $\wh{g}$. This contradicts (iii) and finishes the proof of the theorem.
    \end{proof}

    \begin{remark}
        It is not hard to see that if all the equivalent assertions in Theorem \ref{thm_local_positive_mass} hold, equality in (i) and (ii) can only be achieved if if $g$ is another PE metric. In the case (i), this is just because $g$ must be another critical point of $\mu_{AH,\hat{g}}$. In case (ii), \cite[Theorem 4.5]{poincare_einstein_entropy} implies that $g$ must be of constant scalar curvature $-n(n-1)$ so that $\mu_{AH,\hat{g}}(g)=-m_{{VR},\wh{g}}(g)=0$. Due to case (i), $g$ must be PE. If we have equality in (iii), $m_{{VR},\wh{g}}(g)=2(n-1)RV_{\wh{g}}(g)=0$ and (ii) imply that $g$ must be Einstein. Furthermore, the decay assumption on $g-\wh{g}$ and the Fefferman-Graham expansion for PE metrics even implies that $g$ must be isometric to $\hat{g}$.
    \end{remark}

\section{A \L ojasiewicz--Simon Inequality for the Entropy}\label{loj_section}    
    In this section we prove a \L ojasiewicz--Simon inequality for $\mu_{\mathrm{AH},\wh{g}}$ following the strategy of Colding--Minicozzi \cite{einstein_loj}. Using this approach, a similar \L ojasiewicz--Simon has also recently been established for ALE-manifolds by Deruelle--Ozuch \cite{ale_loj}. 
Again, throughout this section let $(M^n,\wh{g})$ be a fixed PE manifold of class $C^{2,\alpha}$  and $k>n/2+2$.    
     \begin{theorem}\label{loj_general}
       Let $\mathcal{A},\mathcal{B} \subset L^2\left(M\right)$ be closed subspaces such that $H^k\left(M\right) \cap \mathcal{A}$ and $H^{k-2}\left(M\right) \cap \mathcal{B}$ are, respectively, closed subsets of $H^k\left(M\right)$ and $H^{k-2}\left(M\right)$. Assume $\mathcal{F} : \mathcal{U} \rightarrow \mathbb{R}$ is analytic where $\mathcal{U} \subset H^k\left(M\right) \cap \mathcal{A}$ is a neighborhood of $0$. Further assume $\mathcal{F}$ is such that
            \begin{enumerate}
            \item The $L^2$-gradient of $\mathcal{F}$ is a $C^1$ map $\nabla \mathcal{F}: \mathcal{U} \rightarrow H^{k-2}\left(M\right)$ with $\nabla \mathcal{F}\left(0\right) = 0$ and
                \begin{equation*}
                \norm{\nabla \mathcal{F}\left(x\right) - \nabla \mathcal{F}\left(y\right)}_{L^2\left(M\right)} \leq C\norm{x-y}_{H^2\left(M\right)}.
            \end{equation*}
            \item The Fr\'echet derivative of $\nabla \mathcal{F}: \mathcal{U} \rightarrow H^{k-2}\left(M\right)$ is continuous when viewed as a map from $H^2\left(M\right)$ to $L^2\left(M\right)$.
            \item The linearization $L$ of $\nabla \mathcal{F}$ at $0$ is symmetric and bounded from $H^2\left(M\right)$ to $L^2\left(M\right)$ and Fredholm when viewed as a map from $H^2\left(M\right) \cap \mathcal{A}$ to $L^2\left(M\right) \cap \mathcal{B}$.
            \item The previous point holds when $H^2$ and $L^2$ are, respectively, replaced by $H^k$ and $H^{k-2}$.
        \end{enumerate}
            Then there exists a $\theta \in \left(0,1\right]$ and some constant $C > 0$ such that, for sufficiently small $x \in H^2\left(M\right) \cap \mathcal{A}$, we have 
            \begin{equation}
            \abs{\mathcal{F}\left(x\right) - \mathcal{F}\left(0\right)}^{2-\theta} \leq C\norm{\nabla \mathcal{F}\left(x\right)}^2_{L^2\left(M\right)}.
        \end{equation}
    \end{theorem}
    It is worth noting here that the statement of Theorem \ref{loj_general} is more in line with the one found in \cite{einstein_loj} rather than \cite{ale_loj}. In particular, the latter work needs to ensure the $L^2$-kernel of a certain differential operator $L$ is equal to the $L^2$-co-kernel of $L$, which leads the authors of \cite{ale_loj} to deal with a type of weighted H\"older spaces. The authors of \cite{ale_loj} also need to work with these weighted spaces to show $L$ is Fredholm. This all introduces difficulties in showing stability results in a companion paper \cite{der_ozuch_ale_stability}. On the other hand, in our setting of asymptotically hyperbolic manifolds, we are luckily able to bypass these difficulties. This is in large part due to work by Lee \cite{lee_book}, which gives the needed Fredholmness of our stability operator even when dealing with unweighted spaces (see also Remark \ref{fredholm_remark}). This allows us to closely mimic strategies from the compact setting when proving both Theorem \ref{loj_general} and our (in)stability theorems in Section \ref{stability_thms}.
    \begin{remark}
        In a similar vein, one can replace the Sobolev spaces in Theorem \ref{loj_general} with weighted versions $H^k_\delta$, $L^2_\delta$ for a certain class of weights $\delta$ (see \cite{lee_book}). The proof also goes through if one uses weighted H\"older spaces $C^{2,\alpha}_\delta$, $C^{0,\alpha}_\delta$ in appropriate places and changes the theorem to be more in line with \cite{einstein_loj,ale_loj}. That all being said, the current form of Theorem \ref{loj_general} will be sufficient for our purposes. 
    \end{remark}
    Now on to the proof. The following lemmas allow us to apply a Lyapunov--Schmidt type reduction to prove Theorem \ref{loj_general}. These auxiliary results will all implicitly also assume the conditions in the statement of Theorem \ref{loj_general}. Also, we set $\mathcal{K} := \ker_{L^2}\left(L\right)$ and $\mathcal{N} := \nabla \mathcal{F} + \Pi_{\mathcal{K}}$, where $\Pi_{\mathcal{K}}$ is the $L^2$-orthogonal projection onto $\mathcal{K}$. 
        \begin{lemma}\label{inverse_map_lemma}
        There is an open neighborhood of $0 \in  H^{k-2}\left(M\right) \cap \mathcal{B}$, say $\mathcal{O}$, and a map\\ $\Phi: \mathcal{O} \rightarrow H^k\left(M\right) \cap \mathcal{A}$ with $\Phi\left(0\right) = 0$ and $C > 0$ such that for sufficiently small $x,y \in \mathcal{O}$ and $z \in H^k\left(M\right) \cap \mathcal{A}$,
        \begin{enumerate}
            \item[(i)] $\Phi \circ \mathcal{N}\left(z\right) = z$ and $\mathcal{N} \circ \Phi\left(x\right) = x$.
            \item[(ii)] $\norm{\Phi\left(x\right) - \Phi\left(y\right)}_{H^2\left(M\right)} \leq C\norm{x-y}_{L^2\left(M\right)}$ and $\norm{\Phi\left(x\right) - \Phi\left(y\right)}_{H^k\left(M\right)} \leq C\norm{x-y}_{H^{k-2}\left(M\right)}$.
            \item[(iii)] $\mathcal{F} \circ \Phi$ is analytic on $\mathcal{U}$.
        \end{enumerate}
    \end{lemma}
    \begin{proof}
        Note that $\mathcal{N} := \nabla \mathcal{F} + \Pi_{\mathcal{K}}: H^k\left(M\right) \cap \mathcal{A} \rightarrow H^{k-2}\left(M\right)$ is a $C^1$ map by Assumption $(1)$. Furthermore, the Fr\'echet derivative of $\mathcal{N}$ at $0$ is
        \begin{equation*}
            D_0 \mathcal{N} = L + \Pi_{\mathcal{K}}.
        \end{equation*}
       We now want to show $D_0 \mathcal{N}$ is a local isomorphism so that we can use the implicit function theorem to define $\Phi$. We know that $L$ is a Fredholm operator whose index is $0$ due to symmetry and Assumption $(4)$. Since $\mathrm{dim}\left(\mathcal{K}\right) < \infty$, $\Pi_{\mathcal{K}}$ is a compact operator. This implies $D_0\mathcal{N}$ is Fredholm. If we can further show $D_0\mathcal{N}$ is injective then we are done.\\
        Now to show injectivity. Assume for some $x$ that $L\left(x\right) + \Pi_{\mathcal{K}}\left(x\right) = 0$. After projecting onto $\mathcal{K}^\perp$ we find that $L\left(x\right) = 0$ hence $x \in \mathcal{K}$ and $\Pi_{\mathcal{K}}\left(x\right) = 0$. This implies $x = 0$ as desired.\\
        Therefore, an application of the implicit function theorem shows that $D_0 \mathcal{N}$ is an isomorphism from $H^k\left(M\right) \cap \mathcal{A}$ onto $H^{k-2}\left(M\right) \cap \mathcal{B}$. Also, its inverse $\left(D_0\mathcal{N}\right)^{-1}$ is a bounded linear mapping from $H^{k-2}\left(M\right) \cap \mathcal{B}$ to $H^k\left(M\right) \cap \mathcal{A}$. Applying the inverse function theorem gives the existence of $\mathcal{O}$, a neighborhood of $0 \in H^{k-2}\left(M\right) \cap \mathcal{B}$, and a $C^1$ map $\Phi: U \rightarrow H^k\left(M\right) \cap \mathcal{A}$ with the following properties:
        \begin{itemize}
            \item $\Phi\left(0\right) = 0$,
            \item $\Phi \circ \mathcal{N}\left(z\right) = z$ and $\mathcal{N} \circ \Phi\left(x\right) = x$ for $x,y \in \mathcal{O}$ and $z \in H^k\left(M\right) \cap \mathcal{A}$ small enough (thereby establishing (i)),
            \item a continuous Fr\'echet derivative and $D_y \Phi = \left(D_{\Phi\left(y\right)}\mathcal{N}\right)^{-1}$.
        \end{itemize}
    

        This third point allows us to use the integral mean value theorem for Banach spaces to prove (ii):
        for sufficiently small $x,y \in \mathcal{O}$ we have
        \begin{equation*}
            \Phi\left(y\right) - \Phi\left(x\right) = \int^1_0 \left(D\Phi\right)\left(ty + \left(1-t\right)x\right)dt \left(y-x\right)
        \end{equation*}
        and then, using $\norm{Tx}_Y \leq \norm{T}_{\mathrm{op}}\norm{x}_X$ for an operator $T: X \rightarrow Y$ and $\norm{\cdot}_{\mathrm{op}}$ the operator norm,
        \begin{equation*}
            \norm{\Phi\left(y\right) - \Phi\left(x\right)}_{H^k\left(M\right)} \leq \norm{\int^1_0 D\Phi\left(ty + \left(1-t\right)x\right)dt}_{\mathrm{op}}\norm{y-x}_{H^{k-2}\left(M\right)} \leq C\norm{y-x}_{H^{k-2}\left(M\right)}.
        \end{equation*}
       The last line follows from $x,y \in \mathcal{O}$ being sufficiently small and $\Phi$ having a continuous Fr\'echet derivative.\\
        The other Lipschitz estimate, $\norm{\Phi\left(y\right) - \Phi\left(x\right)}_{H^k\left(M\right)} \leq C\norm{x-y}_{L^2\left(M\right)}$, follows from a similar argument. In particular, by Assumption (2), $\nabla \mathcal{F}$ has a continuous Fr\'echet derivative when viewed as a map from $H^2$ to $L^2$. Also, Assumption (3) and the $L^2$-boundedness of $\Pi_{\mathcal{K}}$ tells us, using analogous logic to before, that $D_0 \mathcal{N}$ is a local isomorphism and its inverse is a linear map.\\
        Finally, (iii) (the analyticity of $\mathcal{F} \circ \Phi$) follows from the analyticity of $\mathcal{F}$ and that of $\Phi$, which is assured by the implicit function theorem (see page $1081$ in \cite{whitt_implicit}).
    \end{proof}
    For the next lemma, recall the definition of $\nabla \mathcal{F}$, the $L^2$-gradient of a functional\\ $\mathcal{F}$: $D_y \mathcal{F}\left(v\right) = \left<\nabla \mathcal{F},v\right>_{L^2}$ for $y \in \mathcal{O}$.
    \begin{lemma}\label{comp_estimate}
        There is a constant $C > 0$ such that, for sufficiently small $x \in H^k\left(M\right) \cap \mathcal{A}$, we have
        \begin{equation*}
            \norm{\nabla \mathcal{G}\left(\Pi_{\mathcal{K}}\left(x\right)\right)}_{L^2\left(M\right)} \leq C \norm{\nabla \mathcal{F}\left(x\right)}_{L^2\left(M\right)},
        \end{equation*}
            where $\mathcal{G} := \mathcal{F} \circ \Phi$. Moreover, if $y_t := \Pi_{\mathcal{K}}\left(x\right) + t\nabla \mathcal{F}\left(x\right)$ for $t \in \left[0,1\right]$, then
            \begin{equation*}
            \norm{\nabla \mathcal{G}\left(y_t\right)}_{L^2\left(M\right)} \leq C \norm{\nabla \mathcal{F}\left(x\right)}_{L^2\left(M\right)}.
        \end{equation*}
    \end{lemma}
        \begin{proof}
        For sufficiently small $y \in \mathcal{O}$, the neighborhood of $0$ in $H^k\left(M\right) \cap \mathcal{B}$ from Lemma \ref{inverse_map_lemma}, we have, by the chain rule, $D_y \mathcal{G}\left(v\right) = D_{\Phi\left(y\right)}\mathcal{F} \circ D_y \Phi\left(v\right)$ for $v \in L^2\left(M\right)$. Therefore
            \begin{align*}
            \abs{D_y \mathcal{G}\left(v\right)} &= \abs{\left<\nabla \mathcal{G},v\right>_{L^2}}\\
            &\leq \norm{\nabla \mathcal{F}\left(\Phi\left(y\right)\right)}_{L^2\left(M\right)} \norm{D_y\Phi\left(v\right)}_{L^2\left(M\right)}\\
            &\leq C \norm{\nabla \mathcal{F}\left(\Phi\left(y\right)\right)}_{L^2\left(M\right)} \norm{v}_{L^2\left(M\right)},
        \end{align*}
            where the final line follows from Item (2) of Lemma \ref{inverse_map_lemma}. Dividing through by $\norm{v}_{L^2\left(M\right)}$ and using the definition of $\nabla \mathcal{G}$ yields
           \begin{equation*}
            \norm{\nabla \mathcal{G}\left(y\right)}_{L^2\left(M\right)} \leq C \norm{\nabla \mathcal{F}\left(\Phi\left(y\right)\right)}_{L^2\left(M\right)}.
        \end{equation*}
            In particular, for small enough $x \in H^2\left(M\right) \cap \mathcal{A}$, we have
            \begin{equation*}
            \norm{\nabla \mathcal{G}\left(\Pi_{\mathcal{K}}\left(x\right)\right)}_{L^2\left(M\right)} \leq C \norm{\nabla \mathcal{F}\left(\Phi\left(\Pi_{\mathcal{K}}\left(x\right)\right)\right)}_{L^2\left(M\right)}.
        \end{equation*}
            Furthermore, since $x = \Phi \circ \mathcal{N}\left(x\right) = \Phi \circ \left(\Pi_{\mathcal{K}}\left(x\right) + \nabla \mathcal{F}\left(x\right)\right)$ we have
            \begin{align*}
            &\norm{\nabla \mathcal{F}\left(\Phi\left(\Pi_{\mathcal{K}}\left(x\right)\right)\right) - \nabla \mathcal{F}\left(x\right)}_{L^2\left(M\right)}\\
            = &\norm{\nabla \mathcal{F}\left(\Phi\left(\Pi_{\mathcal{K}}\left(x\right)\right)\right) - \nabla \mathcal{F}\left( \Phi \left(\Pi_{\mathcal{K}}\left(x\right) + \nabla \mathcal{F}\left(x\right)\right)\right)}_{L^2\left(M\right)}\\
            \leq &C \norm{\Phi\left( \Pi_{\mathcal{K}}\left(x\right)\right) - \Phi\left(\Pi_{\mathcal{K}}\left(x\right) + \nabla \mathcal{F}\left(x\right)\right)}_{H^2\left(M\right)}\\
            \leq &C \norm{\nabla \mathcal{F}\left(x\right)}_{L^2\left(M\right)}.
        \end{align*}
            Here the third line follows from Assumption (1) while the final line is due to Lemma \ref{inverse_map_lemma}. The triangle inequality then implies
            \begin{align*}
            \norm{\nabla \mathcal{G}\left(\Pi_{\mathcal{K}}\left(x\right)\right)}_{L^2\left(M\right)} &\leq C\norm{\nabla \mathcal{F}\left(\Phi\left(\Pi_{\mathcal{K}}\left(x\right)\right)\right)}_{L^2\left(M\right)}\\
            &\leq C\norm{\nabla \mathcal{F}\left(x\right)}_{L^2\left(M\right)} + C\norm{\nabla \mathcal{F}\left(\Phi\left(\Pi_{\mathcal{K}}\left(x\right)\right)\right) - \nabla \mathcal{F}\left(x\right)}_{L^2\left(M\right)}\\
            &\leq C\norm{\nabla \mathcal{F}\left(x\right)}_{L^2\left(M\right)}.
        \end{align*}
            For the second part of the lemma, note that $\mathcal{G} = \mathcal{F} \circ \Phi$ and $\Phi \circ \mathcal{N} = \mathrm{id}$ which implies $\mathcal{F} = \mathcal{G}\circ \mathcal{N}$. Therefore
            \begin{equation*}
            D_{\Phi\left(y_t\right)} \mathcal{F} = D_{y_t} \mathcal{G} \circ D_{\Phi\left(y_t\right)} \mathcal{N}.
        \end{equation*}
            Lemma \ref{inverse_map_lemma} tells us $D_{\Phi\left(y_t\right)} \mathcal{N}$ is invertible with bounded inverse. We can thus write\\ $D_{y_t} \mathcal{G} = D_{\Phi\left(y_t\right)}\mathcal{F} \circ \left(D_{\Phi\left(y_t\right)} \mathcal{N}\right)^{-1}$. Using the definition of $L^2$-gradients again, as well as the Cauchy--Schwarz inequality yields
        \begin{equation*}
            \norm{\nabla \mathcal{G}\left(y_t\right)}_{L^2\left(M\right)} \leq C \norm{\nabla \mathcal{F}\left(\Phi\left(y_t\right)\right)}_{L^2\left(M\right)}.
        \end{equation*}
            We now want to refine this estimate so that it holds with $x$ in place of $\Phi\left(y_t\right)$ on the right hand side. For this, note that $y_1 = \Pi_{\mathcal{K}}\left(x\right) + \nabla \mathcal{F}\left(x\right) = \mathcal{N}\left(x\right)$, which means $\Phi\left(y_1\right) = \Phi \circ \mathcal{N}\left(x\right) = x$. This lets us compute as follows:
            \begin{align*}
            \norm{\nabla \mathcal{F}\left(\Phi\left(y_t\right)\right) - \nabla \mathcal{F}\left(x\right)}_{L^2\left(M\right)} &\leq C\norm{\Phi\left(y_t\right) - x}_{H^2\left(M\right)}\\
            &= C\norm{\Phi\left(y_t\right) - \Phi\left(y_1\right)}_{H^2\left(M\right)}\\
            &\leq C\norm{y_t - y_1}_{L^2\left(M\right)}\\
            &= C \norm{t \nabla \mathcal{F}\left(x\right) - \nabla \mathcal{F}\left(x\right)}_{L^2\left(M\right)}\\
            &= C\left(1-t\right)\norm{\nabla \mathcal{F}\left(x\right)}_{L^2\left(M\right)}.
        \end{align*}
        Here the first inequality follows from Assumption (1), the second and fourth lines are due to the definition of $y_t$, and the third comes from Lemma \ref{inverse_map_lemma}. Using the triangle inequality implies the desired result:
        \begin{align*}
            \norm{\nabla \mathcal{G}\left(y_t\right)}_{L^2\left(M\right)} &\leq C \norm{\nabla \mathcal{F}\left(\Phi\left(y_t\right)\right)}_{L^2\left(M\right)}\\
            &\leq C \norm{\nabla \mathcal{F}\left(\Phi\left(y_t\right)\right) - \nabla \mathcal{F}\left(x\right)}_{L^2\left(M\right)} + \norm{\nabla \mathcal{F}\left(x\right)}_{L^2\left(M\right)}\\
            &\leq C\norm{\nabla \mathcal{F}\left(x\right)}_{L^2\left(M\right)}.\qedhere
        \end{align*}
    \end{proof}
    \begin{lemma}\label{almost_loj_estim}
        There is a constant $C > 0$ such that, for sufficiently small $x \in H^k\left(M\right) \cap \mathcal{A}$, we have
        \begin{equation*}
            \abs{\mathcal{F}\left(x\right) - \mathcal{G}\left(\Pi_{\mathcal{K}}\left(x\right)\right)} \leq C\norm{\nabla \mathcal{F}\left(x\right)}^2_{L^2\left(M\right)}.
        \end{equation*}
    \end{lemma}
    \begin{proof}
        For all $t \in \left[0,1\right]$ define $y_t := \Pi_{\mathcal{K}}\left(x\right) + t \nabla \mathcal{F}\left(x\right)$. Then $\Phi\left(y_1\right) = \Phi \circ \mathcal{N}\left(x\right) = x$,\\ $y_0 = \Pi_{\mathcal{K}}\left(x\right)$, and $\frac{d}{dt}y_t = \nabla \mathcal{F}\left(x\right)$. The fundamental theorem of calculus gives
        \begin{align*}
            \int^1_0 \left<\nabla \mathcal{G}\left(y_t\right), \nabla \mathcal{F}\left(x\right)\right>_{L^2} dt &= \int^1_0 \frac{d}{dt} \mathcal{G}\left(y_t\right) dt\\
            &= \mathcal{G}\left(y_1\right) - \mathcal{G}\left(y_0\right)\\
            &= \mathcal{F}\left(\Phi\left(y_1\right)\right) - \mathcal{G}\left(y_0\right)\\
            &= \mathcal{F}\left(x\right) - \mathcal{G}\left(\Pi_{\mathcal{K}}\left(x\right)\right).
        \end{align*}
        The Cauchy--Schwarz inequality then tells us that
        \begin{align*}
            \abs{\mathcal{F}\left(x\right) - \mathcal{G}\left(\Pi_{\mathcal{K}}\left(x\right)\right)} &\leq \int^1_0 \abs{\left<\nabla \mathcal{G}\left(y_t\right), \nabla \mathcal{F}\left(x\right)\right>_{L^2}} dt\\
            &\leq \norm{\nabla \mathcal{G}\left(y_t\right)}_{L^2\left(M\right)} \norm{\nabla \mathcal{F}\left(x\right)}_{L^2\left(M\right)}\\
            &\leq C\norm{\nabla \mathcal{F}\left(x\right)}^2_{L^2\left(M\right)},
        \end{align*}
        where the final line is due to Lemma \ref{comp_estimate}.
    \end{proof}
    Now we can prove Theorem \ref{loj_general}:
    \begin{proof}
        Note that $\mathcal{G}_{\mathcal{K}} := \left.\mathcal{G}\right|_{\mathcal{K}}$, the restriction of $\mathcal{G}$ to the (finite dimensional) kernel $\mathcal{K}$, is analytic. Then, for small enough $x \in H^k\left(M\right) \cap \mathcal{A}$, we can use Lemma \ref{comp_estimate}, the finite dimensional \L ojasiewicz--Simon inequality from \cite{finite_loj}, and $\norm{\nabla \mathcal{G}_{\mathcal{K}}\left(y\right)}_{L^2\left(M\right)} \leq \norm{\nabla \mathcal{G}\left(y\right)}_{L^2\left(M\right)}$ for $y \in \mathcal{K}$, to find that, for $\theta \in \left(0,1\right]$,
        \begin{align*}
            C^2\norm{\nabla \mathcal{F}\left(x\right)}^2_{L^2\left(M\right)} &\geq C \norm{\nabla \mathcal{G}\left(\Pi_{\mathcal{K}}\left(x\right)\right)}^2_{L^2\left(M\right)}\\
            &\geq \norm{\nabla \mathcal{G}_{\mathcal{K}}\left(\Pi_{\mathcal{K}}\left(x\right)\right)}^2_{L^2\left(M\right)}\\
            &\geq \abs{\mathcal{G}_{\mathcal{K}}\left(\Pi_{\mathcal{K}}\left(x\right)\right) - \mathcal{G}_{\mathcal{K}}\left(0\right)}^{2-\theta}\\
            &= \abs{\mathcal{G}_{\mathcal{K}}\left(\Pi_{\mathcal{K}}\left(x\right)\right) - \mathcal{F}\left(0\right)}^{2-\theta}.
        \end{align*}
        Using Lemma \ref{almost_loj_estim} and the triangle inequality implies that
        \begin{align*}
            \abs{\mathcal{F}\left(x\right) - \mathcal{F}\left(0\right)}^{2-\theta} &\leq C\left(\abs{\mathcal{F}\left(0\right) - \mathcal{G}\left(\Pi_{\mathcal{K}}\left(x\right)\right)}^{2-\theta} + \abs{\mathcal{F}\left(x\right) - \mathcal{G}\left(\Pi_{\mathcal{K}}\left(x\right)\right)}^{2-\theta}\right)\\
            &\leq C\left(\norm{\nabla \mathcal{F}\left(x\right)}^{2\left(2-\theta\right)}_{L^2\left(M\right)} + \norm{\nabla \mathcal{F}\left(x\right)}^2_{L^2\left(M\right)}\right)\\
            &\leq C \norm{\nabla \mathcal{F}\left(x\right)}^2_{L^2\left(M\right)}
        \end{align*}
        as desired. Note that to get the final inequality we used the continuity of $\nabla \mathcal{F}$ and $a^{2\left(2-\theta\right)} \leq a^2$ for $0 < a \leq 1$.
    \end{proof}
    In order to ensure Assumption (3) of Theorem \ref{loj_general} is satisfied, we need to understand the image of $\Lap_E\left(H^2\left(M\right) \cap \mathcal{A}\right)$, where $\mathcal{A} := \ker\left(\div_g\right)$. If we had higher regularity, then we would be done by the Fredholmness of $\Lap_E$ and an identity for its commutation with the divergence operator (see, for example, pages $28-29$ in \cite{lichnerowicz_commutators}). On the other hand, when looking at $\Lap_E\left(H^2\left(M\right) \cap \mathcal{A}\right)$, while one would expect that the image is $L^2\left(M\right) \cap \mathcal{A}$, this requires the existence of one derivative, so we need to interpret it in a suitable fashion.
    \begin{lemma}\label{image_lemma}
        Let $\mathcal{K} := \ker_{L^2}\left(\Lap_E\right)$ and $\mathcal{A},\Pi_{\mathcal{K}}$ be defined as above. Then the image of\\ $\left(\Lap_E + \Pi_{\mathcal{K}}\right)\left(H^2\left(M\right) \cap \mathcal{A}\right)$ is $L^2\left(M\right) \cap \div^\ast_g\left(C^\infty_c\left(TM\right)\right)^\perp$.
    \end{lemma}
    \begin{proof}
        Consider $v$ in the image of $\left(\Lap_E + \Pi_{\mathcal{K}}\right)\left(H^2\left(M\right) \cap \mathcal{A}\right)$. Then, for some $H^3$-vector field $X$, we have
        \begin{align*}
            \left<v, \div^\ast_g X\right>_{L^2\left(M\right)} &= \left<\Lap_E h + \Pi_{\mathcal{K}}h, \div^\ast_g X\right>_{L^2\left(M\right)}\\
            &= \left<\Lap_L h + 2\left(n-1\right)h + \Pi_{\mathcal{K}} h, \div^\ast_g X\right>_{L^2\left(M\right)}\\
            &= \left<h, \Lap_L \div^\ast_g X\right>_{L^2\left(M\right)} - \left<\div_g\left(2\left(n-1\right) h + \Pi_{\mathcal{K}}h\right), X\right>_{L^2\left(M\right)}\\
            &= \left<h, \div^\ast_g \Lap_H X\right>_{L^2\left(M\right)}\\
            &= -\left<\div_g h, \Lap_H X\right>_{L^2\left(M\right)}\\
            &= 0.
        \end{align*}
        Here $\Lap_H$ is the Hodge Laplacian and $\Lap_L$ is the Lichnerowicz Laplacian. Note that the second line is due to the relation $\Lap_E = \Lap_L + 2\left(n-1\right)$, while the fourth line is due to $\mathcal{K} \subset \ker\left(\div_g\right) $ and $h \in \mathcal{A} = \ker\left(\div_g\right)$, as well as the commutation relation between $\div^\ast$ and $\Lap_L$ from pages $28-29$ in \cite{lichnerowicz_commutators}.
    \end{proof}
    Now to check that the conditions of Theorem \ref{loj_general} are met in our setting.
    We first prove the following lemma, which controls the first order variation of $f_g$:
    \begin{lemma}\label{first_order_var_control}
         For any $g \in \mathcal{R}^k\left(M,\wh{g}\right)$ which is $\eps$-close to $\wh{g}$ for some sufficiently small $\eps > 0$, we have
        \begin{equation*}
            \norm{f'}_{H^2\left(M\right)} \leq C\norm{h}_{H^2\left(M\right)},
        \end{equation*}
        where $f' := \delta_s f\left(h\right)$ is the first order variation of $g \mapsto f_g$ in the direction $h$, while\\ $C = C\left(n,\eps,\wh{g}\right) > 0$ is a constant.
    \end{lemma}
    \begin{proof}
        We first consider the variation of \eqref{el_equation_entropy}, the Euler--Lagrange equation for the entropy $\mu_{\mathrm{AH},\wh{g}}\left(g\right)$:
        \begin{equation*}
            2\Lap_g f_g + \abs{\nabla^g f_g}^2_g - \Sc_g - n\left(n-1\right) + 2\left(n-1\right)f_g = 0.
        \end{equation*}
        We refer the reader to Lemma $5.12$ in \cite{poincare_einstein_entropy} for the derivation. Taking the variation of this Euler--Lagrange equation yields
        \begin{align*}
            2\Lap_g f' + 2\left(n-1\right)f' = &-2\left<h,\nabla^{g,2} f_g\right> - 2\left<\div_g\left(h\right), \nabla^g f_g\right> + \left<\div_g\left(\tr_g\left(h\right)\right), \nabla^g f_g\right>\\
            &- h\left(\nabla^g f_g, \nabla^g f_g\right) + 2\left<\nabla_g f_g, \nabla^g f'\right> - \div^2_g\left(h\right) + \Lap_g \tr_g\left(h\right) + \left<h, \Ric\left(g\right)\right>.
        \end{align*}
        The needed variational formulas can be found in the proofs of Proposition $5.17$ and Proposition $5.22$ in \cite{poincare_einstein_entropy}.\\
        Since $\Lap + \left(n-1\right) : H^2\left(M\right) \rightarrow L^2\left(M\right)$ is an isomorphism by Lemma \ref{isomorphism_result}, we have the following estimate:
        \begin{equation*}
            \norm{f'}_{H^2\left(M\right)} \leq C\norm{\Lap_g f' + \left(n-1\right)f'}_{L^2\left(M\right)}
        \end{equation*}
        for some uniform constant $C > 0$. Since we have assumed $\frac{n}{2} + 2 < k$, we have $\norm{f}_{C^2\left(M\right)} < \eps' < 1$ for some $\eps' < \eps$ by the Sobolev embedding theorem and $g$ being $\eps$-close to $\wh{g}$ in the $H^k$-sense. Combining this and the estimate from above with the variation of the Euler--Lagrange equation, we have
        \begin{align*}
            \norm{f'}_{H^2\left(M\right)} &\leq C\left(n,\eps,\wh{g}\right)\norm{h}_{H^2\left(M\right)} + 2\norm{\left<\nabla^g f_g,\nabla^g f'\right>}_{L^2\left(M\right)}\\
            &\leq C\left(n,\eps,\wh{g}\right)\norm{h}_{H^2\left(M\right)} + \eps'\norm{\nabla^g f'}_{L^2\left(M\right)}.
        \end{align*}
        Absorption then implies the desired result: $\norm{f'}_{H^2\left(M\right)} \leq C\left(n,\eps,\wh{g}\right)\norm{h}_{H^2\left(M\right)}$.
    \end{proof}
    Now to prove the main estimates of this subsection.
    \begin{proposition}\label{loj_estimate_check}
          For any $g_1,g_2 \in \mathcal{R}^k\left(M,\wh{g}\right)$ which are each $\eps$-close to $\wh{g}$ for some sufficiently small $\eps > 0$, we have
        \begin{equation}\label{entropy_grad_estim}
            \norm{\nabla \mu_{\mathrm{AH},\wh{g}}\left(g_2\right) - \nabla \mu_{\mathrm{AH},\wh{g}}\left(g_1\right)}_{L^2\left(M\right)} \leq C\left(n,\eps,\wh{g}\right)\norm{g_2 - g_1}_{H^2\left(M\right)}.
        \end{equation}
        Under the same assumptions we have, for any $h \in H^2\left(M\right)$,
        \begin{equation}\label{entropy_grad_frechet_estim}
            \norm{D_{g_2}\nabla \mu_{\mathrm{AH},\wh{g}}\left(h\right) - D_{g_1}\nabla \mu_{\mathrm{AH},\wh{g}}\left(h\right)}_{L^2\left(M\right)} \leq C\left(n,\eps,\wh{g}\right)\norm{h}_{H^2\left(M\right)}.
        \end{equation}
    \end{proposition}
    \begin{proof}
        The $L^2$-gradient of $\mu_{\mathrm{AH},\wh{g}}$ is
        \begin{equation*}
            \nabla \mu_{\mathrm{AH},\wh{g}}\left(g\right) = \left(\Ric\left(g\right) + \nabla^{g,2} f_g + \left(n-1\right)g\right)e^{-f_g}.
        \end{equation*}
        Therefore, we initially have the following estimate:
        \begin{equation*}
            \begin{split}
                &\abs{\nabla \mu_{\mathrm{AH},\wh{g}}\left(g_2\right) - \nabla \mu_{\mathrm{AH},\wh{g}}\left(g_1\right)} \\
                &\leq \abs{\left(\Ric\left(g_2\right) + \nabla^{g_2,2} f_{g_2} + \left(n-1\right)g_2\right)e^{-f_{g_2}} - \left(\Ric\left(g_1\right) + \nabla^{g_1,2} f_{g_1} + \left(n-1\right)g_1\right)e^{-f_{g_1}}}.
            \end{split}                    
        \end{equation*}
        Taylor expanding reduces the proof of \eqref{entropy_grad_estim} to proving the following estimates, where\\ $h := g_2 - g_1$,
        \begin{align*}
            \norm{\Ric\left(g_2\right) - \Ric\left(g_1\right)}_{L^2\left(M\right)} &\leq C\left(n,\eps,\wh{g}\right)\norm{h}_{H^2\left(M\right)}\\ 
            \norm{\nabla^{g_2,2} f_{g_2} - \nabla^{g_1,2} f_{g_1}}_{L^2\left(M\right)} &\leq C\left(n,\eps,\wh{g}\right)\norm{h}_{H^2\left(M\right)}
        \end{align*}             
        For the first inequality, we can get a more convenient, albeit long, expression using the following formula for the difference $\Ric\left(g_2\right) - \Ric\left(g_1\right)$:
        \begin{align*}
            -2\Ric\left(g_2\right) + 2\Ric\left(g_1\right) = g^{-1}_2 \ast \nabla^{g_1,2} h + g^{-1}_2 \ast h \ast \Rm\left(g_1\right) + g^{-1}_2 \ast g^{-1}_2 \ast \nabla^{g_1} h \ast \nabla^{g_1} h - \mathscr{L}_B\left(g_2\right),
        \end{align*}
        where $B\left(g_1,g_2\right) := \div_{g_1}\left(g_2 - g_1\right) - \frac{1}{2}\nabla^{g_1}\tr_{g_1}\left(g_2 - g_1\right) + g^{-1}_2 \ast \left(g_2 - g_1\right) \ast \nabla^{g_1}\left(g_2 - g_1\right)$. Details about the derivation can be found in multiple places, for instance the proof of Lemma $2.1$ in \cite{shi}.\\
        We will need the following identity for the Lie derivative of a symmetric $2$-tensor $T$:
        \begin{equation}\label{metric_lie_deriv}
            \mathscr{L}_X \left(T\right) = T_{ik} \nabla_j X^i + T_{ij} \nabla_k X^i + X^i \nabla_i T_{kj}.
        \end{equation}
        Also, note that for any tensor $T$ one has
        \begin{equation}\label{metric_grad_swap}
            \nabla^{g_2} T = \nabla^{g_1}T + g^{-1}_1 \ast \nabla^{g_1} h \ast T,
        \end{equation}
        which follows from looking at the difference of the Levi-Civita connections, hence the difference of the associated Christoffel symbols. In the end, we get the following estimate:
        \begin{equation*}
            \norm{\Ric\left(g_2\right) - \Ric\left(g_1\right)}_{L^2\left(M\right)} \leq C\left(\eps, \wh{g}\right)\norm{h}_{H^2\left(M\right)}.
        \end{equation*}

        Note that we also used $\abs{\Rm\left(g_1\right)} < \infty$ and $\norm{h}_{C^2\left(M\right)} < \infty$ by the Sobolev embedding theorem.\\
        Now for the difference of Hessians. We can rewrite them as Lie derivatives and manipulate as follows (where $g_t := g_1 + \left(t-1\right)h$):
        \begin{align*}
            \nabla^{g_2,2}f_{g_2} - \nabla^{g_1,2}f_{g_1} &= \mathscr{L}_{\nabla^{g_2}f_{g_2}}\left(g_2\right) - \mathscr{L}_{\nabla^{g_1}f_{g_1}}\left(g_1\right)\\
            &= \int^2_1 \frac{d}{ds}\mathscr{L}_{\nabla^{g_s}f_{g_s}}\left(g_s\right) ds\\
            &= \int^2_1 \mathscr{L}_{\nabla^{g_s}f_{g_s}}\left(h\right) - \mathscr{L}_{h\left(\nabla^{g_s} f_{g_s}\right)}\left(g_s\right) + \mathscr{L}_{\nabla^{g_s}f'}\left(g_s\right)ds.
        \end{align*}          
        Using \eqref{metric_lie_deriv} and the boundedness of $\nabla_{g_s}f_{g_s}$, we get
        \begin{align*}
            \abs{\mathscr{L}_{\nabla^{g_s}f_s}\left(h\right)} &\leq 2\abs{h}\abs{\nabla^{g_s,2} f_{g_s}} + \abs{\nabla^{g_s} h} \abs{\nabla^{g_s} f_{g_s}} \leq C\left(n,\eps,\wh{g}\right)\left(\abs{h} + \abs{\nabla^{g_1}h}\right)\\
            \abs{\mathscr{L}_{h\left(\nabla^{g_s}f_{g_s}\right)}\left(g_s\right)} &\leq 2\left(\abs{\nabla^{g_s} h}\abs{\nabla^{g_s} f_{g_s}} + \abs{h} \abs{\nabla^{g_s,2} f_{g_s}}\right)\leq C\left(n,\eps,\wh{g}\right)\left(\abs{h} + \abs{\nabla^{g_1}h}\right).
        \end{align*}
        Note that to get the final inequality in each line we also used \eqref{metric_grad_swap} with $g_2 = g_s$ and $T = h$. To deal with the final term $\mathscr{L}_{\nabla^{g_s}f'}\left(g_s\right)$, we can use Lemma \ref{first_order_var_control} to get
       \begin{equation*}
            \norm{\mathscr{L}_{\nabla^{g_s}f'}\left(g_s\right)}_{L^2\left(M\right)} \leq C\norm{\nabla^{g_s,2}f'}_{L^2\left(M\right)} \leq C\norm{f'}_{H^2\left(M\right)} \leq C\norm{h}_{H^2\left(M\right)}.
        \end{equation*}
        Putting everything together and using $h = g_2 - g_1$ yields \eqref{entropy_grad_estim} as desired.\\
        As for \eqref{entropy_grad_frechet_estim}, note first that we have
        \begin{equation*}
            \begin{split}
                D_g \nabla \mu_{\mathrm{AH},\wh{g}}\left(h\right) &= D_g\left(\left(\Ric\left(g\right) + \nabla^{g,2} f_g + \left(n-1\right) g\right)e^{-f_g}\right)\\
                &= \left[\frac{1}{2}\Lap_L h + \mathscr{L}_{\div_g\left(h\right)} g - \nabla^{g,2} \tr_g\left(h\right) + \left(n-1\right)h\right.\\
                &+ \left.\left(n-1\right)h + \left(\frac{1}{2}\tr_g\left(h\right) - f'\right) \left(\Ric\left(g\right) + \nabla^{g,2} f_g + \left(n-1\right)g\right)\right.\\
                &+ \left.\nabla^{g,2} f' - \frac{1}{2}\left(\nabla_i h_{jk} + \nabla_j h_{ik} - \nabla_k h_{ij}\right) \nabla_k f\right].
            \end{split}
        \end{equation*}
        Using this expression along with the estimates and procedures from above, we have
        \begin{equation*}
            \norm{D_{g_2}\nabla \mu_{\mathrm{AH},\wh{g}}\left(h\right) - D_{g_1}\nabla \mu_{\mathrm{AH},\wh{g}}\left(h\right)}_{L^2\left(M\right)} \leq C\left(n,\eps,\wh{g}\right)\norm{h}_{H^2\left(M\right)}
        \end{equation*}
        for a variation $h \in H^2\left(M\right)$, which is what we wanted.
    \end{proof}
    We can now prove the desired inequality.
    \begin{theorem}\label{loj_ineq_thm}
        There is a neighborhood $\mathcal{U} \subset \mathcal{R}^k\left(M,\wh{g}\right)$ of $\wh{g}$, in which the entropy $\mu_{\mathrm{AH},\wh{g}}$ satisfies the following: there exist a positive constant $C > 0$ and an exponent $\theta \in \left(0,1\right]$ such that we have
        \begin{equation}\label{loj_ineq}
            \abs{\mu_{\mathrm{AH},\wh{g}}\left(g\right) - \mu_{\mathrm{AH},\wh{g}}\left(\wh{g}\right)}^{2-\theta} \leq C\norm{\nabla\mu_{\mathrm{AH},\wh{g}}\left(g\right)}^2_{L^2\left(M\right)}.
        \end{equation}
    \end{theorem}
    \begin{proof}
        By the observations made in Section \ref{sec_entropy}, we know that the relative entropy is defined on a neighborhood $\mathcal{U} \subset \mathcal{R}^k\left(M,\wh{g}\right)$ of $\wh{g}$. We now further take $\mathcal{U}$ small enough so that Lemma \ref{first_order_var_control} and Proposition \ref{loj_estimate_check} hold on it. Note that both sides of \eqref{loj_ineq} are diffeomorphism invariant and recall that the discussion immediately before Proposition \ref{entropy_analytic_extension} tells us $\mu_{\mathrm{AH},\wh{g}}$ is analytic on $\mathcal{U}$. Next, denote the neighborhood from Proposition \ref{entropy_analytic_extension} by $\mathcal{U}_{\mathrm{slice}}$. Note that, possibly after shrinking $\mathcal{U}_{\mathrm{slice}}$, we can assume $ \mathcal{U}_{\mathrm{slice}} \subset \mathcal{U}$. Proposition \ref{entropy_analytic_extension} tells us we can restrict the domain of $\mu_{\mathrm{AH},\wh{g}}$ to $\mathcal{U}_{\mathrm{slice}} \subset \mathcal{R}^k\left(M,\wh{g}\right) \cap \ker\left(\div_{\wh{g}}\right)$. We also restrict the image, by Lemma \ref{image_lemma}, to $H^{k-2}\left(M\right) \cap \div^\ast_{\wh{g}}\left(C^\infty_c\left(TM\right)\right)^\perp$. Denote this restriction by $\mu^{\mathrm{res}}_{\mathrm{AH},\wh{g}}$. Note that it is still analytic and
        \begin{equation*}
            \norm{\nabla \mu^{\mathrm{res}}_{\mathrm{AH},\wh{g}}}_{L^2\left(M\right)} \leq \norm{\nabla \mu_{\mathrm{AH},\wh{g}}}_{L^2\left(M\right)}.
        \end{equation*}
        This tells us Conditions $(1)$ and $(2)$ in Theorem \ref{loj_general} are satisfied by $\mu^{\mathrm{res}}_{\mathrm{AH},\wh{g}}$ since they are satisfied by $\mu_{\mathrm{AH},\wh{g}}$ due to Proposition \ref{loj_estimate_check} and Corollary $5.18$ from \cite{poincare_einstein_entropy}.\\
        Conditions $(3)$ and $(4)$ in Theorem \ref{loj_general} are satisfied due Proposition $5.22$ in \cite{poincare_einstein_entropy}, which tells us the linearization of $\nabla \mu^{\mathrm{res}}_{\mathrm{AH},\wh{g}}$ at $\wh{g}$ is $-\frac{1}{2}\Lap_{E,\wh{g}}$. This is symmetric and Fredholm between $H^2$ and $L^2$, as well as between $H^k$ and $H^{k-2}$, as discussed in Remark \ref{fredholm_remark}. Furthermore, boundedness follows from the definition of the Sobolev norms:
        \begin{align*}
            \norm{\Lap_{E,\wh{g}} h}_{L^2\left(M\right)} &\leq C\left(n,\eps,\wh{g}\right)\norm{h}_{H^2\left(M\right)}\\
            \norm{\Lap_{E,\wh{g}} h}_{H^{k-2}\left(M\right)} &\leq C\left(n,\eps,\wh{g}\right)\norm{h}_{H^k\left(M\right)}.\qedhere
        \end{align*}
    \end{proof}
    \begin{remark}
        While the assumption about $g$ lying in a sufficiently small $H^k$-neighborhood of $\wh{g}$ might seem too strong, it is naturally satisfied in our setting. In particular, a result of Bamler (see Lemma \ref{rdtf_hl_bounds}) tells us that the Ricci-DeTurck flow starting in a sufficiently small $L^2 \cap L^\infty$-neighborhood of $\wh{g}$ will eventually lie in an $H^k$-neighborhood of $\wh{g}$, for $k$ as large as one wants.
    \end{remark}

\section{Proof of the Stability and Instability Theorems}\label{stability_thms}
    Before proving our stability and instability theorems, we record some results due to Bamler (\cite{bamler_symm_spaces}) that will be useful. Firstly, it will be convenient for us to consider the Ricci-DeTurck flow:
    \begin{equation}\label{rdtf}
        \begin{cases}
            \partial_t g\left(t\right) = -2\left(\Ric_{g\left(t\right)} + \left(n-1\right)g\left(t\right) + \frac{1}{2}\mathscr{L}_{X_{\wh{g}}\left(g\left(t\right)\right)}\left(g\left(t\right)\right)\right)\\
            g_0 = g\left(0\right).
        \end{cases}
    \end{equation}
    Here $X_{\wh{g}}\left(h\right)$ is a vector field defined as, for any $h \in S^2 T^\ast M$,
    \begin{equation*}
        X_{\wh{g}}\left(h\right) := \div_{\wh{g}}\left(h\right) + \frac{1}{2}\nabla^{\wh{g}}\tr_{\wh{g}}\left(h\right).
    \end{equation*}
    Proposition $2.4$ from \cite{bamler_symm_spaces} gives short time existence and $L^\infty$-bounds on the perturbation $h\left(t\right):=g\left(t\right)-\wh{g}$:
    \begin{proposition}\label{rdtf_short_time}
        Let $\left(M,\wh{g}\right)$ be a complete Riemannian manifold with its curvature tensor globally bounded in the $C^{0,\alpha}$-sense and let $g_0=\wh{g}+h_0$  be a smooth metric on $M$. Then there are constants $\eps,\tau > 0$, $0 < C < \infty$ depending only on $M$ and $\wh{g}$ such that, if $\norm{h_0}_{L^\infty\left(M\right)} < \eps$, there is a unique $L^\infty$-bounded smooth solution $\left\{h\left(t\right)\right\}_{t \in \left[0,\tau\right]}$ to \eqref{rdtf}. Moreover, we have
        \begin{equation*}
            \norm{h\left(t\right)}_{L^\infty\left(M \times \left[0,\tau\right]\right)} \leq C\norm{h_0}_{L^\infty\left(M\right)}.
        \end{equation*}
    \end{proposition}
    Lemma $6.2$ from \cite{bamler_symm_spaces} then gives $H^\ell$-control on the perturbation, for $\ell \geq 0$ as large as one likes. We have rephrased it slightly to fit in our setting better.
    \begin{lemma}\label{rdtf_hl_bounds}
        Let $\left(M,\wh{g}\right)$ be a PE manifold of class $C^{2,\alpha}$ and let $g_0=\wh{g}+h_0$ be a smooth metric on $M$. Then there are constants $A_\ell,\eps_\ell > 0$ such that for all $\ell, b \geq 0$ the following holds. Let $\left\{h_t\right\}_{t \in \left[0,\tau\right]}$ be a solution to the Ricci-deTurck flow with $\abs{h_0} < \eps_\ell$ and $\norm{h_0}_{L^2\left(M\right)} \leq b$. Then for all $t \in \left[0,\tau\right]$, with $\tau > 0$ the existence time from Proposition \ref{rdtf_short_time}, and $\ell = 0$, or $\ell \geq 1$ and $t \in \left[\frac{\tau}{2},\tau\right]$, we have
        \begin{equation*}
            \norm{h\left(t\right)}_{H^\ell\left(M\right)} \leq A_\ell b.
        \end{equation*}
        Also, for all $t \in \left[\frac{\tau}{2},\tau\right]$ we have the pointwise estimates
        \begin{equation*}
            \abs{\nabla^\ell h\left(t\right)} \leq A_\ell b.
        \end{equation*}
    \end{lemma}
    Now we can prove our stability result Theorem \ref{mainthm_stability}, which we restate below for the reader's convenience. Here, and in the proof, all the norms and function spaces are defined with respect to the background metric $\wh{g}$ unless otherwise specified. 
    \begin{theorem}\label{pe_stability}
%
 Let $\left(M,\wh{g}\right)$ be a Poincar\'e--Einstein manifold of class $C^{2,\alpha}$.  If $\wh{g}$ is a local maximizer of the expander entropy $\mu_{\mathrm{AH},\wh{g}}$, then for any $L^2\cap L^{\infty}$-neighborhood $  \mathcal{U} $ of $\wh{g}$ there exists another $L^2\cap L^{\infty}$-neighborhood $\mathcal{V} \subset \mathcal{U}$ of  $\wh{g}$ with the following property:\\
Any Ricci flow $g_t$ starting from some $g_0 \in \mathcal{V}$ exists for all time and there exists a family of diffeomorphisms $\varphi_t$ such that $\varphi_t^*g_t$ stays in $\mathcal{U}$ for all time and converges in all derivatives to a  Poincar\'e--Einstein metric $g_{\infty}$ as $t\to\infty$.
      The convergence is polynomial, i.e.\ for all $k\in\N$, there exist constants $C, \beta > 0$ such that
        \begin{equation*}
            \norm{\wh{g} - g_\infty}_{C^k} \leq C\left(t+1\right)^{-\beta}
        \end{equation*}
        for all $t \geq 1$.
    \end{theorem}
    \begin{proof}
        Let $\mathcal{U}$ be an $\eps$-neighbhorhood of $\wh{g}$. For any $g_0 \in \mathcal{V}$, consider the Ricci de Truck flow starting from $g_0$. Proposition \ref{rdtf_short_time} gives existence of \eqref{rdtf} for all $t \in \left[0,\tau\right]$, for some $\tau < \infty$.
        
   Let $k>n/2+2$ and $\mathcal{B}_{r}$ be the ball of radius $r$ around $\wh{g}$ with respect to the $H^k$-norm. Let $\epsilon>0$ be so small
that $\mathcal{B}_{\epsilon}\subset\mathcal{U}$ and such 
    that $\mu_{\mathrm{AH},\wh{g}}(g)\leq \mu_{\mathrm{AH},\wh{g}}(\wh{g})$  and Theorem \ref{loj_ineq_thm} holds for all metrics $g\in \mathcal{B}_{\epsilon}$. Furthermore, we may also assume that we can apply Lemma \ref{rdtf_hl_bounds} to all $g_0\in \mathcal{B}_{\epsilon}$.
        
    By  Lemma \ref{rdtf_hl_bounds}, we can take $\mathcal{V}$ so small that the Ricci-de Turck flow $g(t)$ starting at $g_0\in \mathcal{V}$ stays in $\mathcal{U}$ for $t\in [0,\tau]$ and in $\mathcal{B}_{\epsilon/2}$ for $t\in [\tau/2,\tau]$.
%
      Let $T' \geq \tau$ be the maximal time such that any Ricci-DeTurck flow ${g\left(t\right)}$ starting in $\mathcal{V}$ exists and there exists a family of diffeomorphisms $\varphi_{t}$ such that $\varphi_t^*g(t)$ stays within $\mathcal{B}_{\epsilon}$ for all $t \in \left[\tau,T'\right)$.\\
        By the definition of $T'$, we have
        \begin{equation*}
            \sup_M \abs{\Rm_{g\left(t\right)}}_{g\left(t\right)} \leq C_0
        \end{equation*}
        for all $t \in \left[\tau,T'\right)$. Lemma \ref{rdtf_hl_bounds} implies that for all $\ell \in \mathbb{N}$, there is a constant $C_{\ell}$ such that
        \begin{equation}\label{shi_bounds}
            \sup_M \abs{\nabla^\ell \Rm_{g\left(t\right)}}_{g\left(t\right)} \leq C_\ell
        \end{equation}
        for all $t \in \left[\tau,T'\right)$. Then, as $f_g$ satisfies the elliptic equation
        \begin{equation*}
            2\Lap_g f_g + \abs{\nabla f_g}^2_g - \Sc_g - n\left(n-1\right) + 2\left(n-1\right) f_g = 0,
        \end{equation*}
        Corollary \ref{cor_Lap_iso} tells us that
        \begin{equation}\label{minimizer_bounds}
            \sup_M \abs{\nabla^\ell f_{g\left(t\right)}}_{g\left(t\right)} \leq C'_\ell
        \end{equation}
        for all $t \in \left[\tau,T'\right)$. Note that all the above estimates are diffeomorphism invariant.\\
        Now we modify the flow for $t \geq \tau$ to be more suited for the use of our \L ojasiewicz--Simon inequality from Theorem \ref{loj_ineq_thm}. For this, consider the family of diffeomorphisms $\left\{\varphi_t\right\}_{t \in \left[\tau,T'\right)}$ generated by $-    X_{\wh{g}}\left(g\left(t\right)-\wh{g}\right)-\nabla^{g\left(t\right)}f_{g\left(t\right)}$ with $\varphi_\tau = \mathrm{id}_M$. For $t \in \left[\tau,T'\right)$, set $\wt{g}\left(t\right) := \varphi^\ast_t g\left(t\right)$ and $\wt{h}\left(t\right) := \varphi^\ast_t g\left(t\right) - \wh{g}$. This yields a \emph{modified} Ricci flow. In particular, \eqref{minimizer_bounds} allows us to replicate the arguments in \cite{bamler_cusps,bamler_symm_spaces} to deduce that Proposition \ref{rdtf_short_time} and Lemma \ref{rdtf_hl_bounds} still hold for the modified flow. Let $\wt{T} \in \left(\tau,T'\right]$ be the maximal time of such that this new flow exists and stays within $\mathcal{U}$. Then, for all $t \in \left[\tau,\wt{T}\right)$,  we have
        \begin{equation*}
            \partial_t \wt{h}\left(t\right) = \partial_t \wt{g}\left(t\right) = -2\left(\Ric_{\wt{g}\left(t\right)} + \left(n-1\right)\wt{g}\left(t\right) + \nabla^2 f_{\wt{g}\left(t\right)}\right).
        \end{equation*}
        In all, we now have the following:
        \begin{equation*}
            \wt{g}\left(t\right) = \begin{cases}
                g\left(t\right)~~&\mathrm{for}~t \in \left[0,\tau\right]\\
                \varphi^\ast_t g\left(t\right)~~&\mathrm{for}~t \in \left[\tau,\wt{T}\right)
            \end{cases}
        \end{equation*}
        and analogously for $\wt{h}\left(t\right)$. 
        By interpolation, (see e.g.\ \cite[Corollary 12.7]{hamilton_interpolation}), we get for every small $\eta > 0$ a large
         $\ell \in \mathbb{N}$ and a constant $C' > 0$ such that
        \begin{equation}\label{stability_interp_bounds}
            \norm{\partial_t \wt{h}\left(t\right)}_{H^k} \leq \norm{\partial_t \wt{h}\left(t\right)}^\eta_{H^\ell} \norm{\partial_t \wt{h}\left(t\right)}^{1-\eta}_{L^2} \leq C' \norm{\partial_t \wt{h}\left(t\right)}^{1-\eta}_{L^2}.        \end{equation}
 Note that the second inequality follows from Lemma \ref{rdtf_hl_bounds}. 
Let  $\theta$ be the exponent in the \L ojasiewicz--Simon inequality from Theorem \ref{loj_ineq_thm} and choose $\eta>0$ such that we have  $\sigma := \theta - \eta + \theta \eta > 0$. Note that $\sigma - 1 = \left(\theta - 1\right)\left(1-\eta\right)$. Then Lemma $5.20$ from \cite{poincare_einstein_entropy}, as well as \eqref{stability_interp_bounds} and Theorem \ref{loj_ineq_thm} tell us
        \begin{align*}
            \frac{d}{dt}\mu_{\mathrm{AH},\wh{g}}\left(\wt{g}\left(t\right)\right) &= \norm{\Ric_{\wt{g}\left(t\right)} + \left(n-1\right)\wt{g}\left(t\right) + \nabla^2 f_{\wt{g}\left(t\right)}}^{1+\eta}_{L^2} \norm{\partial_t\wt{h}\left(t\right)}^{1-\eta}_{L^2}\\
            &\geq \frac{1}{C'}\norm{\Ric_{\wt{g}\left(t\right)} + \left(n-1\right)\wt{g}\left(t\right) + \nabla^2 f_{\wt{g}\left(t\right)}}^{1+\eta}_{L^2} \norm{\partial_t\wt{h}\left(t\right)}_{H^k}\\
            &\geq \frac{1}{C'}\abs{\mu_{\mathrm{AH},\wh{g}}\left(\wt{g}\left(t\right)\right)}^{\left(1-\theta\right)\left(1+\eta\right)}\norm{\partial_t\wt{h}\left(t\right)}_{H^k}.
        \end{align*}
        Then, since $\mu_{\mathrm{AH},\wh{g}}\left(\wt{g}\left(t\right)\right) \leq \mu_{\mathrm{AH},\wh{g}}\left(\wh{g}\right)=0$ by assumption, we can use Theorem \ref{loj_ineq_thm} and the inequality we just derived to deduce that
        \begin{align*}
            -\frac{d}{dt}\abs{\mu_{\mathrm{AH},\wh{g}}\left(\wt{g}\left(t\right)\right)}^\sigma &= \sigma \abs{\mu_{\mathrm{AH},\wh{g}}\left(\wt{g}\left(t\right)\right)}^{\sigma-1} \frac{d}{dt} \mu_{\mathrm{AH},\wh{g}}\left(\wt{g}\left(t\right)\right)\geq \frac{\sigma}{C'}\norm{\partial_t \wt{h}\left(t\right)}_{H^k}.
        \end{align*}
        Integrating this from $\tau$ to $\wt{T}$ yields
        \begin{equation}\label{loj_stability_application}
            \int^{\wt{T}}_{\tau} \norm{\partial_t \wt{h}\left(t\right)}_{H^k\left(M\right)} dt \leq \frac{C'}{\sigma}\left(\abs{\mu_{\mathrm{AH},\wh{g}}\left(\wt{g}\left(\tau\right)\right) - \mu_{\mathrm{AH},\wh{g}}\left(\wh{g}\right)}^\sigma\right) < \frac{\eps}{4}, 
        \end{equation}
        possibly after further shrinking $\mathcal{V}$. Applying the triangle inequality then yields
        \begin{equation*}
            \norm{\wh{g} - \wt{g}\left(\wt{T}\right)}_{H^k} \leq \norm{\wh{g} - \wt{g}\left(\tau\right)}_{H^k} + \norm{\wt{g}\left(\tau\right) - \wt{g}\left(\wt{T}\right)}_{H^k} < \frac{\eps}{2}.
        \end{equation*}
        We therefore must have $\wt{T} = T' = T = \infty$, which implies $\wt{g}\left(t\right) \rightarrow g_\infty$ for some limit metric $g_\infty \in \mathcal{U}$. Note that this is also due to another use of Lemma \ref{rdtf_hl_bounds} to ensure $\norm{\wt{h}\left(t\right)}_{L^\infty\left(M\right)} < \frac{\eps}{4}$.
         \\
        Using Theorem \ref{loj_ineq_thm} again, one computes that, for $t \geq \tau$,
        \begin{equation*}
            \frac{d}{dt} \abs{\mu_{\mathrm{AH},\wh{g}}\left(\wt{g}\left(t\right)\right) - \mu_{\mathrm{AH},\wh{g}}\left(\wh{g}\right)}^{\theta-1} \geq -C''
        \end{equation*}
        for some constant $C'' > 0$. Integrating this with respect to time yields, for $t \geq 1$,
        \begin{equation}\label{entropy_conv_rate}
            \abs{\mu_{\mathrm{AH},\wh{g}}\left(\wt{g}\left(t\right)\right) - \mu_{\mathrm{AH},\wh{g}}\left(\wh{g}\right)} \leq C''\left(t+1\right)^{-\frac{1}{1-\theta}}.
        \end{equation}
        Sending $t \rightarrow \infty$ tells us $\mu_{\mathrm{AH},\wh{g}}\left(g_\infty\right) = \mu_{\mathrm{AH},\wh{g}}\left(\wh{g}\right)$. Since $\wh{g}$ is a local maximum of the entropy, so is $g_\infty$. Corollary $5.18$ from \cite{poincare_einstein_entropy} then tells us $g_\infty$ must be a Poincar\'e--Einstein metric.\\
        Finally, combining \eqref{loj_stability_application} and \eqref{entropy_conv_rate} yields the claimed convergence rate:
        \begin{equation*}
            \norm{\wt{g}\left(t\right) - \wh{g}}_{H^k} \leq C\abs{\mu_{\mathrm{AH},\wh{g}}\left(\wt{g}\left(t\right)\right) - \mu_{\mathrm{AH},\wh{g}}\left(\wh{g}\right)}^\sigma \leq C'''\left(t+1\right)^{-\frac{\sigma}{1-\theta}},
        \end{equation*}
        for a constant $C''' > 0$ and $t \geq \tau$. Using Lemma \ref{rdtf_hl_bounds} and Sobolev embedding yields the desired convergence in all derivatives.
    \end{proof}
    Now for the instability theorem:
    \begin{theorem}\label{pe_instability}
        Let $\left(M,\wh{g}\right)$ be a Poincar\'e--Einstein manifold of class $C^{2,\alpha}$. Assume that $\wh{g}$ is \emph{not} a local maximizer of the relative entropy $\mu_{\mathrm{AH},\wh{g}}$. Then there exists a non-trivial ancient Ricci flow $\left\{g\left(t\right)\right\}_{t \in \left(-\infty,0\right]}$ and a family of diffeomorphisms $\left\{\varphi_t\right\}_{t \in \left(-\infty,0\right]}$  such that $\varphi_t^*g(t)\to\wh{g}$ in all derivatives as $t \rightarrow -\infty$.
    \end{theorem}
    \begin{proof}
        Consider $\left\{g_i\right\} \subset \mathcal{R}^k\left(M,\wh{g}\right)$, a sequence of smooth metrics on $M$ such that $g_i \rightarrow \wh{g}$ in the $L^2 \cap L^\infty$-sense. Furthermore, suppose $\mu_{\mathrm{AH},\wh{g}}\left(g_i\right) > \mu_{\mathrm{AH},\wh{g}}\left(\wh{g}\right)$ for all $i$ and let $\wt{g}_i\left(t\right)$ (resp. $\wt{h}_i\left(t\right)$) be the modified Ricci-DeTurck flow, as defined in the proof of Theorem \ref{pe_stability}, starting at $g_i$ (resp. $h_i := g_i - \wh{g}$). Note that $\wt{g}_i\left(\tau\right)$ converges to $\wh{g}$ in the $H^k$-sense by Lemma \ref{rdtf_hl_bounds}.\\
        Now, take a sufficiently small neighborhood $\mathcal{U} \subset \mathcal{R}^k\left(M,\wh{g}\right)$ of $\wh{g}$ so that there is an $\eps > 0$ such that any $g \in \mathcal{U}$ is $2\eps$-close to $\wh{g}$ and Theorem \ref{loj_ineq_thm} holds on $\mathcal{U}$. Since Lemma \ref{rdtf_hl_bounds} tells us $\wt{g}_i\left(t\right) \in \mathcal{U}$ for all $i \gg 1$ and 
$t \geq \tau$, we can apply Theorem \ref{loj_ineq_thm} to deduce
        \begin{equation*}
            \frac{d}{dt}\left(\mu_{\mathrm{AH},\wh{g}}\left(\wt{g}_i\left(t\right)\right) - \mu_{\mathrm{AH},\wh{g}}\left(\wh{g}\right)\right)^{\theta-1} \geq -C_1
        \end{equation*}
        for all $i \gg 1$ and $t \geq \tau$, where $\theta$ is the exponent in Theorem \ref{loj_ineq_thm} and $C_1 > 0$ is a constant. Integrating this inequality in time, we find that
        \begin{equation}\label{entropy_diff_bound}
            \left[\left(\mu_{\mathrm{AH},\wh{g}}\left(\wt{g}_i\left(t\right)\right) - \mu_{\mathrm{AH},\wh{g}}\left(\wh{g}\right)\right)^{\theta-1} - C_1\left(s-t\right)\right]^{-\frac{1}{1-\theta}} \leq \mu_{\mathrm{AH},\wh{g}}\left(\wt{g}_i\left(s\right)\right) - \mu_{\mathrm{AH},\wh{g}}\left(\wh{g}\right),
        \end{equation}
        provided $\wt{g}_i\left(t\right) \in \mathcal{U}$. Therefore, there exists some time $t_i$ such that
        \begin{equation*}
            \norm{\wt{g}_i\left(t_i\right) - \wh{g}}_{H^k\left(M\right)} = \eps.
        \end{equation*}
        Note that $t_i \rightarrow \infty$ as $i \rightarrow \infty$, otherwise $\wt{g}_i\left(t_i\right) \rightarrow \wh{g}$ in the $H^k$-sense. Using the interpolation estimates from Corollary $12.7$ in \cite{hamilton_interpolation} like before, we have for any given $\eta > 0$ that 
        \begin{equation*}\label{instability_interp_bounds}
            \norm{\partial_t \wt{h}_i\left(t\right)}_{H^k} \leq \norm{\partial_t \wt{h}_i\left(t\right)}^\eta_{H^\ell} \norm{\partial_t \wt{h}_i\left(t\right)}^{1-\eta}_{L^2} \leq C_2 \norm{\partial_t \wt{h}_i\left(t\right)}^{1-\eta}_{L^2}
        \end{equation*}
for some $\ell\in\N$ and a constant $C_2 > 0$. Note that the $H^\ell$-norm is finite due to Lemma \ref{rdtf_hl_bounds}. Let $\eta>0$ be so small that $\sigma := \theta - \eta + \theta \eta > 0$. Then we can compute similarly to the proof of Theorem \ref{pe_stability} to get
        \begin{equation*}
            \frac{d}{dt}\left(\mu_{\mathrm{AH},\wh{g}}\left(\wt{g}_i\left(t\right)\right) - \mu_{\mathrm{AH},\wh{g}}\left(\wh{g}\left(t\right)\right)\right)^\sigma \geq \frac{\sigma}{C_2}\norm{\partial_t \wt{h}_i\left(t\right)}_{H^k}.
        \end{equation*}
        Integrating this new bound with respect to time from $t = \tau$ to $t = t_i$ allows us to deduce that, since $\mu_{\mathrm{AH},\wh{g}}\left(\wt{g}_i\left(t\right)\right) > \mu_{\mathrm{AH},\wh{g}}\left(\wh{g}\right)$,
        \begin{equation}\label{non_trivial_rf_estim}
            \eps = \norm{\wt{g}_i\left(t_i\right) - \wh{g}}_{H^k} \leq \norm{\wt{g}_i\left(1\right) - \wh{g}}_{H^k} + C_2\left(\mu_{\mathrm{AH},\wh{g}}\left(\wt{g}_i\left(t_i\right)\right) - \mu_{\mathrm{AH},\wh{g}}\left(\wh{g}\right)\right)^\sigma.
        \end{equation}
        Now we define a new sequence of metrics by applying a time shift. Let $\wt{g}^s_i\left(t\right) := \wt{g}_i\left(t + t_i\right)$ for $t \in \left[T_i,0\right]$ where $T_i := \tau - t_i \rightarrow -\infty$ as $i \rightarrow \infty$. Then, by the previous parts of the proof, we have
        \begin{align*}
            \norm{\wt{g}^s_i\left(t\right) - \wh{g}}_{H^k\left(M\right)} &\leq \eps\\
            \wt{g}^s_i\left(T_i\right) \rightarrow \wh{g},
        \end{align*}
        where $t \in \left[T_i,0\right]$ and the convergence is in the $H^k$-sense. Additionally, note that we can apply Hamilton's compactness theorem to pass to a convergent subsequence $\wt{g}^s_i\left(t\right) \rightarrow \wt{g}\left(t\right)$ where $\wt{g}\left(t\right)$ is an ancient Ricci flow and $t \in \left(-\infty,0\right]$. Furthermore,
        \begin{equation*}
            \partial_t \wt{g}\left(t\right) = -2\left(\Ric_{\wt{g}\left(t\right)} + \left(n-1\right)\wt{g}\left(t\right) + \nabla^2 f_{\wt{g}\left(t\right)}\right).
        \end{equation*}
        Next, for $t \in \left(-\infty,0\right]$, let $\varphi_t$ be the family of diffeomorphisms generated by the family of vector fields $-    X_{\wh{g}}\left(\wt{g}\left(t\right)-\wh{g}\right)-\nabla f_{\wt{g}\left(t\right)}$ and $\varphi_0 = \mathrm{id}_M$. Then $g\left(t\right) := \varphi^\ast_t \wt{g}\left(t\right)$ is a Ricci flow. Additionally, sending $i \rightarrow \infty$ in \eqref{non_trivial_rf_estim} tells us
        \begin{equation*}
            \eps \leq C_2\left(\mu_{\mathrm{AH},\wh{g}}\left(g\left(0\right)\right) - \mu_{\mathrm{AH},\wh{g}}\left(\wh{g}\right)\right)^\theta,
        \end{equation*}
        hence $g\left(t\right)$ is a non-trivial Ricci flow. We may then use Theorem \ref{loj_ineq_thm} and \eqref{entropy_diff_bound} to find that, for $t \in \left[T_i, 0\right]$,
        \begin{align*}
            \norm{\wt{g}^s_i\left(T_i\right) - \wt{g}^s_i\left(t\right)}_{H^k} &\leq C_3\left(\mu_{\mathrm{AH},\wh{g}}\left(\wt{g}_i\left(t+t_i\right)\right) - \mu_{\mathrm{AH},\wh{g}}\left(\wh{g}\right)\right)^\sigma\\
            &\leq C_3\left(-C_1t + \left(\mu_{\mathrm{AH},\wh{g}}\left(\wt{g}_i\left(t_i\right)\right) - \mu_{\mathrm{AH},\wh{g}}\left(\wh{g}\right)\right)^{\theta-1}\right)^{-\frac{\sigma}{1-\theta}}\\
            &\leq \left(-C_4t + C_5\right)^{-\frac{\sigma}{1-\theta}},
        \end{align*}
        where $C_3,C_4,C_5 > 0$ are constants. The triangle inequality therefore yields
        \begin{equation*}
            \norm{\wh{g} - \wt{g}\left(t\right)}_{H^k} \leq \norm{\wh{g} - \wt{g}^s_i\left(T_i\right)}_{H^k} + \left(-C_4 t + C_5\right)^{-\frac{\sigma}{1-\theta}} + \norm{\wt{g}^s_i\left(t\right) - \wt{g}\left(t\right)}_{H^k}.
        \end{equation*}
        Since $\wt{g}^s_i\left(T_i\right) \rightarrow \wh{g}$ and $\wt{g}^s_i\left(t\right) \rightarrow \wt{g}\left(t\right)$ as $i \rightarrow \infty$, we find that by sending $t \rightarrow -\infty$ we have $\wt{g}\left(t\right) \rightarrow \wh{g}$ in the $H^k$-sense. Hence $\left(\varphi^{-1}_t\right)^\ast g\left(t\right) \rightarrow \wh{g}$ in the $H^k$-sense as $t \rightarrow -\infty$. Convergence in all derivatives follows from Lemma \ref{rdtf_hl_bounds} and Sobolev embedding.
    \end{proof}

\end{document}